\documentclass[reqno,oneside,12pt]{amsart}

\usepackage[T1]{fontenc}
\usepackage{times,mathptm}
\usepackage{amssymb,epsfig,verbatim,xypic}
\theoremstyle{plain}

\newtheorem{thm}{Theorem}[section]

\newtheorem{cor}[thm]{Corollary}
\newtheorem{pro}[thm]{Proposition}
\newtheorem{lem}[thm]{Lemma}
\newtheorem{proposition-principale}[thm]{Proposition principale}
\newtheorem{thm-principal}{Main Theorem}
\newtheorem{defi}[thm]{Definition}
\newtheorem{nota}[thm]{Notation}

\theoremstyle{definition}

\newtheorem{eg}[thm]{Example}
\newtheorem{rem}[thm]{Remark}

\newenvironment{defi-G}
{\noindent{\bf Definition.}\it}{\\}

\newenvironment{thm-M}
{\noindent{\bf Main Theorem.}\it }{}

\newenvironment{thm-C}
{\noindent{\bf Classification Theorem.}\it }{}

\newenvironment{thm-AA}
{\noindent{\bf Theorem A'.}\it}{\\ }

\newenvironment{thm-B}
{\noindent{\bf Theorem B.}\it}{\\ }

\newenvironment{thm-BB}
{\noindent{\bf Theorem B'.}\it}

\def\C{\mathbf{C}}
\def\R{\mathbf{R}}
\def\Q{\mathbf{Q}}
\def\H{\mathbf{H}}
\def\Z{\mathbf{Z}}
\def\N{\mathbf{N}}

\def\ii{{\sf{i}}}

\def\P{\mathbb{P}}

\def\disk{\mathbb{D}}

\def\U{{\mathcal{U}}}
\def\V{{\mathcal{V}}}
\def\F{{\mathcal{F}}}

\def\zbar{{\overline{z}}}
\def\wbar{{\overline{w}}}

\def\L{\mathcal{L}}
\def\K{\mathcal{K}}
\def\A{\mathcal{A}}
\def\B{\mathcal{B}}
\def\E{\mathcal{E}}

\def\Aut{{\sf{Aut}}}
\def\Leb{{\sf{Leb}}}

\def\Pes{{\mathcal{P}}}
\def\Dens{{\mathrm{Dens}}}
\def\Isom{{\sf{Iso}}}

\def\NS{{\mathrm{NS}}}
\def\PC{{\mathrm{PC}}}
\def\NN{{\mathrm{N}}}
\def\Fat{{\mathrm{Fat}}}

\def\vol{{\rm{vol}}}

\def\PGL{{\sf{PGL}}\,}

\def\GL{{\sf{GL}}\,}

\newcommand{\Lip}{{\rm Lip \, }}
\newcommand{\Id}{{\rm Id}}

\def\Pic{{\mathrm{Pic}}}

\def\dist{{\sf{dist}}}

\addtocounter{section}{0}             
\numberwithin{equation}{section}       

\begin{document}

\setlength{\baselineskip}{0.56cm}        
%
%
\title[Dynamics on complex surfaces]
{Automorphisms of surfaces:\\
 Kummer rigidity and measure of maximal entropy}
\date{2013/2014}
\author{Serge Cantat and Christophe Dupont}
\address{IRMAR (UMR 6625 du CNRS)\\ 
Universit{\'e} de Rennes 1 
\\ France}
\email{serge.cantat@univ-rennes1.fr}
\email{christophe.dupont@univ-rennes1.fr}

%
%

%
%

%
%

\begin{abstract} 

We classify complex projective surfaces with an automorphism of positive entropy for which the
unique invariant measure of maximal entropy is absolutely continuous with respect to Lebesgue measure.

\vspace{0.15cm}

\noindent{\sc{R\'esum\'e.}} Nous classons les surfaces complexes  projectives munies d'un automorphisme
dont l'unique mesure d'entropie maximale est absolument continue par rapport \`a la mesure de Lebesgue. 

\end{abstract}

\maketitle

\setcounter{tocdepth}{1}
\tableofcontents

\section{Introduction}

\subsection{Automorphisms and absolutely continuous measures} \label{autoabs}

Let $X$ be a complex projective surface and $\Aut(X)$ be the group of holomorphic diffeomorphisms of 
$X$. By Gromov-Yomdin theorem, the topological entropy ${\sf{h}}_{top}(f)$ of every $f \in \Aut(X)$
is equal to the logarithm of the spectral radius $\lambda_f $ of the linear endomorphism
\[
f^*\colon H^2(X;\Z)\to H^2(X;\Z)
\]
(here $H^2(X;\Z)$ is the second cohomology group of $X$). Thus, $f$ has positive
entropy if, and only if there is an eigenvalue $\lambda\in \C$ of $f^*$ with $\vert \lambda\vert >1$. 
In fact such an eigenvalue is unique if it exists, and is equal to the spectral radius $\lambda_f$. The number $\lambda_f$ is called the {\bf{dynamical degree}} of $f$. As a root 
of the characteristic polynomial of $f^* \colon H^2(X;\Z)\to H^2(X;\Z)$, it is an algebraic integer; more precisely, $\lambda_f$ is equal to $1$, to a reciprocal quadratic integer, or to a Salem number (see \cite{Cantat:Milnor}).

\subsubsection*{}
When the entropy is positive, there is a natural $f$-invariant probability measure $\mu_f$ on $X$ which satisfies the following properties (see \cite{BLS1, Cantat:Acta}):
\begin{itemize}
\item $\mu_f$ is the unique $f$-invariant probability measure with maximal entropy;
\item if $\mu_n$ denotes the average on the set of isolated fixed points
of $f^n$,   then $\mu_n$ converges towards $\mu_f$ 
as $n$ goes to $+\infty$. 
\end{itemize} 

Thus $\mu_f$ encodes the most interesting features of the dynamics of $f$. In the sequel we shall call $\mu_f$ the {\bf{measure of maximal entropy}} of $f$. 

The main goal of this paper is to study the regularity properties of this measure. By ergodicity, $\mu_f$ is either {\bf{singular}} or {\bf{absolutely continuous}} with respect to Lebesgue measure. By definition, it is singular if there exists a Borel subset $A$ of $X$ satisfying $\mu_f(A)=1$ and $\vol(A)=0$ (the volume is taken with respect to any smooth volume form on $X$); it is absolutely continuous if $\mu_f(B)=0$ for  every Borel subset $B\subset X$ such that ${\vol}(B)=0$. 

\subsubsection*{}
Classical examples of pairs $(X,f)$ for which $\mu_f$ is absolutely continuous are described in Section~\ref{exampleswith} below. The first examples are linear Anosov automorphisms of complex tori; one derives new examples from these linear automorphisms by performing equivariant quotients under a finite group action and by blowing up periodic orbits. Our main theorem, stated in Section~\ref{MT}, establishes that these are the only possibilities. This result answers a question raised by Curtis T. McMullen and by the first author.

\subsubsection*{}
This theorem allows to exhibit automorphisms of complex projective surfaces for which $\mu_f$ is singular (see \S \ref{appli}). All previously known examples were constructed on rational surfaces whereas here, we focus on K3 surfaces. A complex projective surface $X$ is a {\bf K3 surface} if it is simply connected and if it supports a holomorphic $2$-form $\Omega_X$ that does not vanish. Such a form is unique
up to multiplication by a non-zero complex number; thus, if one imposes the constraint 
\[
\int_X\Omega_X\wedge {\overline{\Omega_X}} = 1,
\]
the volume form $\vol_X := \Omega_X\wedge {\overline{\Omega_X}}$ is uniquely determined by the complex
structure of $X$; in particular, this volume form is $\Aut(X)$-invariant. A byproduct of our main theorem is a characterization of the pairs $(X,f)$ for which $\mu_f=\vol_X$; this occurs if and only if $\mu_f$ is absolutely continuous.

\begin{eg}A good example to keep in mind is the family of (smooth) surfaces of degree $(2,2,2)$ in 
$\P^1_\C\times \P^1_\C\times \P^1_\C$. Such a surface $X$ comes with three
double covers $X\to \P^1_\C\times \P^1_\C$, hence with three holomorphic involutions $\sigma_1$, 
$\sigma_2$,  and $\sigma_3$. The composition $f=\sigma_1\circ\sigma_2\circ \sigma_3$ 
is an automorphism of $X$ of positive entropy. The measure $\mu_f$ is singular for a generic choice of $X$ but coincides with $\vol_X$ for specific choices.  \end{eg}

\subsection{Examples with absolutely continuous maximal entropy measure}\label{exampleswith}

\subsubsection{Abelian surfaces} \label{absurface}

 Let $A$ be a complex abelian surface and let $\vol_A$ denote the Lebesgue (i.e. Haar) measure on $A$, normalized by $\vol_A(A)=1$. Every $f \in \Aut(A)$ preserves $\vol_A$, and the measure of maximal entropy $\mu_f$ is equal to $\vol_A$ when $\lambda_f>1$. 
 
 Complex abelian surfaces with automorphisms of positive
entropy have been classified in \cite{Ghys-Verjovsky}. The simplest example is obtained as follows. Start with an 
elliptic curve $E=\C/\Lambda_0$ and consider the product $A=E\times E$. The group $\GL_2(\Z)$  acts
on $\C^2$ linearly, preserving the lattice $\Lambda= \Lambda_0\times \Lambda_0$ ; thus, it acts also on the
quotient $A=\C^2/\Lambda$. This gives rise to a morphism $M\mapsto f_M$ from $\GL_2(\Z)$ 
to $\Aut(A)$. The spectral radius of $(f_M)^*$ on $H^2(A;\Z)$ is equal to the
square of the spectral radius of $M$. In particular, $\lambda_f>1$ as soon as the trace of $M$
satisfies $\vert {\text{tr}}(M)\vert > 2$.

\subsubsection{Classical Kummer surfaces} 

Consider the complex abelian surface $A=E\times E$ as in Section~\ref{absurface}. 
The center of $\GL_2(\Z)$ is generated by the involution $\eta= -{\text{Id}}$; it acts on $A$ by 
\[
\eta(x,y)=(-x,-y).
\]
The quotient $A/\eta$ is a singular surface. Its singularities are sixteen ordinary double points; they can be resolved
by a simple blow-up, each singular point giving rise to a smooth rational curve with self-intersection $-2$. Denote
by $X$ this minimal regular model of $A/\eta$. Since $\GL_2(\Z)$ commutes to $\eta$, one gets
an injective morphism $M\mapsto g_M$ from $\PGL_2(\Z)$ to $\Aut(X)$. The topological entropy of
$g_M$ (on X) is equal to the topological entropy of $f_M$ (on $A$). 
The holomorphic $2$-form $\Omega_A=dx\wedge dy$
is $\eta$-invariant and determines a non-vanishing holomorphic $2$-form $\Omega_X$ on $X$. The volume
form $\Omega_X\wedge {\overline{\Omega_X}}$ is invariant under each automorphism $g_M$, and the associated
probability measure
coincides with the measure of maximal entropy $\mu_{g_M}$ when ${\sf{h}}_{top}(g_M)>0$. 
Hence, again, one gets examples of automorphisms for which the measure of maximal entropy is absolutely continuous. The surface $X$ is a Kummer surface and provides a famous example of K3 surface (see \cite{BPVDVH}). 

\begin{rem} Kummer surfaces $A/\eta$ with $A$ a complex torus and $\eta(x,y)=(-x,-y)$ form a dense subset
of the moduli space of K3 surfaces (see \cite{BPVDVH}). Moreover, there are explicit families $(X_t, f_t)_{t \in \disk}$ of automorphisms of K3 surfaces such that  $(X_0,f_0)$ is such a Kummer example, but $X_t$ stops to be a Kummer surface for $t\neq 0$ (see \cite{Cantat:Panorama}, Section 8.2). 
\end{rem}

\subsubsection{Rational quotients} 

Consider the complex abelian surface $A=E\times E$ of Section~\ref{absurface} given by the lattice $\Lambda_0=\Z[\tau]$
with $\tau^2=-1$ or $\tau^3=1$ (and $\tau\neq 1$). The group $\GL_2(\Z[\tau])$
acts on $A$ and its center contains  
\[
\eta_\tau(x,y)=(\tau x, \tau y).
\] 
The quotient space $X_0=A/\eta_\tau$ is singular and rational. Resolving the singularities, 
one gets examples of smooth rational surfaces $X$ with automorphisms of positive 
entropy. The image of the Lebesgue measure on $A$ provides a probability measure
on $X$ which is smooth on a Zariski open subset of $X$ but has ``poles'' along 
the exceptional divisor of the projection $\pi\colon X\to X_0$. Nevertheless, this measure 
is absolutely continuous with respect to Lebesgue measure (see \cite{Cantat:Panorama}).

\subsection{Main theorem}\label{MT}

The examples of Section~\ref{exampleswith} lead to the following definition (see also \cite{Cantat:Compositio, Cantat-Zeghib, Zhang:2013}). 

\begin{defi} Let $X$ be a complex projective surface and let $f \in \Aut(X)$. The pair $(X,f)$ is a {\bf{Kummer example}} if there exist
\begin{itemize}
\item a birational morphism $\pi : X \to X_0$ onto an orbifold $X_0$, 
\item a finite orbifold cover $\epsilon : Y \to X_0$ by a complex torus $Y$, 
\item an automorphism $f_0$ of $X_0$ and an automorphism $\hat f $ of $Y$ such that 
$$ f_0 \circ \pi = \pi \circ f \quad \textrm{ and } \quad  f_0 \circ \epsilon = \epsilon \circ \hat f   . $$
\end{itemize}  
\end{defi}

If $(X,f)$ is a Kummer example with $\lambda_f>1$, then $\mu_f$ is absolutely continuous with respect
to Lebesgue measure (it is real analytic with integrable poles along a finite set of algebraic curves).

\vspace{0.2cm}

\begin{thm-M}\label{thm:main1}
Let $X$ be a complex projective surface and $f$ be an automorphism of $X$ with positive entropy. Let $\mu_f$ be the measure of maximal entropy of $f$. If $\mu_f$ is absolutely continuous with respect to Lebesgue measure, then $(X,f)$ is a Kummer example. 
\end{thm-M}

\vspace{0.2cm}

This answers a question raised by the first author in his thesis \cite{Cantat:Thesis}
and solves Conjecture 3.31 of Curtis T. McMullen in \cite{McMullen:Algebra-and-Dynamics}. 
Moreover, the surfaces $X$ which can occur are specified by the following easy theorem (see \cite{Cantat-Favre, Cantat-Favre:II}).

\vspace{0.2cm}

\begin{thm-C}\label{lem:CFKummer}
Let $X$ be a complex projective surface and $f$ be an automorphism of $X$ of positive entropy.  
Assume that $(X,f)$ is a Kummer example. 
\begin{enumerate}
\item $X$ is an abelian surface, a K3 surface, or a rational surface. 
\item If $X$ is a K3 surface, it is a classical Kummer surface, i.e. the quotient of an abelian 
surface $A$ by the involution $\sigma(x,y)=(-x,-y)$; in particular, the Picard number of 
$X$ is not less than $17$.
\item If $X$ is rational, then $\lambda_f$ is contained in  $\Q(\zeta_l)$ where $\zeta_l$ is a primitive
root of unity of order $l=3$, $4$, or $5$.
\end{enumerate}  
\end{thm-C}

\vspace{0.2cm}

In particular, $X$ is not an {\bf{Enriques surface}}, i.e. a quotient of K3 surface by a fixed point free holomorphic involution.

\begin{rem}[see \cite{Berteloot-Dupont,BL,Mayer,Zdunik}] An analogous result holds for holomorphic endomorphisms $g$ of the projective space $\P^k_\C$ of topological degree $>1$. In this case, 
there is also an invariant probability measure $\mu_g$ that describes the repartition of periodic points and is the unique measure of maximal entropy, and if $\mu_g$ is absolutely continuous with respect to Lebesgue measure, then $g$ is a {Latt\`es example}: it lifts to an endomorphism of an abelian variety via an equivariant ramified cover. This is due to Zdunik for $k=1$ and to Berteloot, Loeb and the second author for $k \geq 2$. \end{rem}

\begin{rem}[see \cite{Cantat:Compositio}]\label{latt}
There are examples of non injective rational transformations $h\colon X\dasharrow X$ of K3 surfaces such that
\begin{itemize}
\item the topological entropy of $h$ is positive, 
\item $h$ preserves a unique measure of maximal entropy $\mu_h$,
\item $\mu_h$  coincides with the canonical volume form $\vol_X=\Omega_X\wedge {\overline{\Omega_X}}$ on the K3 surface $X$, 
\end{itemize}
but $h$ is not topologically conjugate to a Kummer example. In particular the Kummer and Latt\`es rigidities do not extend to non injective rational mappings.
\end{rem}

\subsection{Applications} \label{appli}


\subsubsection{Lyapunov exponents and Hausdorff dimension} The first consequence relies on theorems of Ledrappier, Ruelle and Young.

\begin{cor}\label{cor1} Let $X$ be a complex projective surface and $f$ be an automorphism of $X$ with positive entropy $\log \lambda_f$. 
Let $\lambda_s < \lambda_u$ denote the negative and positive Lyapunov exponents of the measure $\mu_f$.
The following properties are equivalent 
 \begin{enumerate}
\item $\mu_f$ is absolutely continuous with respect to Lebesgue measure;
\item $\lambda_s  = -{1 \over 2} \log \lambda_f$ and $\lambda_u  =  {1 \over 2} \log \lambda_f$;
\item for $\mu_f$-almost every $x \in X$,
\[
\lim_{r\to 0}\frac{\log \mu_f (B_x(r))}{\log r}=4;
\]
\item $(X,f)$ is a Kummer example.
\end{enumerate} 
\end{cor}

If $X$ is a K3 surface or, more generally, if there is an  $f$-invariant volume form on $X$,  the Lyapunov exponents of $\mu_f$  are opposite ($\lambda_u = - \lambda_s$), and one can replace the second item by either one of the two equalities. 

\subsubsection{K3 and Enriques surfaces} 
Recall that the {\bf{N\'eron-Severi group}} $\NS(X)$ of a complex projective surface $X$ is 
the subgroup of the second homology group $H_2(X;\Z)$  generated by the homology classes of algebraic curves on $X$.
The rank of this abelian group is the {\bf{Picard number}} $\rho(X)$. When $\rho(X)$ is equal to $1$,  the
entropy of every automorphism $f$ of $X$ vanishes (see~\cite{Cantat:Panorama, Cantat:Milnor}); thus, the first interesting case is $\rho(X)=2$.

\begin{cor}\label{cor:K3Pic2} Let $X$ be a complex projective K3 surface with Picard number~$2$. Assume that the 
 intersection form does not represent $0$ and $-2$ on $\NS(X)$. Then 
 \begin{enumerate}
 \item $\Aut(X)$ contains an infinite cyclic subgroup of index at most $2$;
 \item if $f\in \Aut(X)$ has infinite order its entropy is positive and its measure of maximal entropy $\mu_f$ is
singular with respect to Lebesgue measure.
\end{enumerate}\end{cor}

Examples of such K3 surfaces with an infinite group of automorphisms are described in \cite{Wehler, Cantat-Oguiso}.

An Enriques surface $Z$ is the quotient of a K3 surface by a fixed point free involution. If
$\pi\colon X\to Z$ is such a quotient, the  canonical volume form $\vol_X$ of $X$ determines a smooth $\Aut(Z)$-invariant
volume form $\vol_Z$ on $Z$. Thus, Enriques surfaces have a natural invariant volume form that is uniquely determined by
the complex structure. 

\begin{cor}\label{cocor}
If $f$ is an automorphism of an Enriques surface with positive entropy,
then $\mu_f$ is singular with respect to Lebesgue measure. 
\end{cor}

Let now $Z$ be a general Enriques surface. Up to finite index, the group $\Aut(Z)$ is
isomorphic to the group of isometries of the lattice ${\mathbb{U}}\oplus (-{\mathbb{E}}_8)$ (see \cite{ Barth-Peters:1983,Dolgachev-Cossec:Book}), so that it contains automorphisms with positive topological entropy. Thus, Corollary~\ref{cocor} implies that general Enriques surfaces have automorphisms for which the measure of maximal entropy is singular; on the other hand, the volume form $\vol_Z$ provides the only probability measure that is invariant under the action of $\Aut(Z)$ (\cite{Cantat:Trans}).

\subsubsection{Dynamical degrees and rational surfaces}
For the next statement, recall that the dynamical degree $\lambda_f$ is an algebraic integer (see Section \ref{autoabs}).

\begin{cor}\label{thm:dyna-deg}
Let $X$ be a complex projective surface and $f$ be an automorphism of $X$ with positive entropy.  If the degree of  $\lambda_f$ (as an algebraic integer)
is larger than or equal to $5$ then the measure of maximal entropy $\mu_f$ is singular with respect
to Lebesgue measure. 
\end{cor}

This can be applied to the examples constructed by Bedford and Kim (see \cite{Bedford-Kim:2010,Bedford-Kim:2012}) and McMullen (see \cite{McMullen:2007}). They exhibit families of automorphisms of rational surfaces $f_m\colon X_m\to X_m$ for which the degree of the algebraic number $\lambda_{f_m}$ increases with $m$. Thus Corollary~\ref{thm:dyna-deg} shows that $\mu_{f_m}$ is singular with respect to Lebesgue measure when $m$ is large enough. This may also be applied to a family of examples constructed by Blanc. (See Sections~\ref{par:egbdk} and \ref{par:Blanc-auto}  
for more details)

\subsection{Related problems}

\subsubsection{Geodesic flows} A similar question of entropy rigidity concerns the geodesic flow $(\theta_t)_{t\in \R}$ on a negatively curved
riemannian manifold $(M, g)$. Negative curvature implies that this flow is Anosov with a unique invariant probability 
measure $\nu$ of maximal entropy (i.e. with metric entropy ${\sf{h}}(\theta_1, \nu)$ equal to the topological entropy ${\sf{h}}_{top}(\theta_1)$). 
The flow preserves also the Liouville measure $\lambda_g$, and the {\sl{Entropy conjecture}} predicts that $\nu_g$ is
absolutely continuous with respect to $\lambda_g$ if  and only if the riemannian manifold $(M,g)$ is locally symmetric. This is proved by Katok for surfaces (see \cite{Katok}). We refer to \cite{Ledrappier:ICM} for a nice survey on this type of problem 
and to \cite{Ghys:ENS, BFL} for its relationship to the
rigidity properties of Anosov flows  with smooth stable and unstable foliations. 

\subsubsection{Random walks} Another related question concerns the regularity of harmonic measures. Consider the fundamental group $\Gamma_g$ of an orientable, closed surface of genus $g\geq 2$, and
identify the boundary  $\partial \Gamma_g$ to the unit circle ${\mathbb{S}}^1$.
Let  $\nu_S$ be a probability measure on $\Gamma_g$  with finite 
support $S$, such that $S$ generates $\Gamma_g$. The measure $\nu_S$ determines a random walk 
on $\Gamma_g$. Given a starting point $x$ in $\Gamma_g$ and a subset $A$ of the boundary $\partial \Gamma_g$, the harmonic measure $\omega_x(A)$ is the probability that a random path which starts at $x$ converges to a point of
$A$ when time goes to $+\infty$. It is conjectured that $\omega_x$ is singular with respect to  Lebesgue measure
on $\partial \Gamma_g= {\mathbb{S}}^1$. We refer to   \cite{Kaimanovich-LePrince, BHM} for an introduction to this topic
and references; see also \cite{Bourgain} for a recent example.

\subsection{Organization of the paper}

Fix a complex projective surface $X$ and an automorphism $f$ of $X$ with positive entropy. 
Section \ref{par:prelim} is devoted to classical facts concerning the dynamics of $f$. In particular, we explain that  $\mu_f$ is the product of two closed positive currents $T^+_f$ and $T^-_f$, that 
the generic stable and unstable manifolds of $f$ are parametrized by
holomorphic entire curves $\xi:\C\to X$, and that $T^+_f$ and $T^-_f$  are respectively  
obtained by integration on these stable and unstable manifolds. 
Assume, now that $\mu_f$ is absolutely continuous with respect to Lebesgue measure. 

\subsubsection{Renormalization along the invariant manifolds} We first show that 
\begin{itemize}
\item[(1)] the absolute continuity of 
$\mu_f$ can be transferred to the currents $T^+_f$ and~$T^-_f$; 
\item[(2)] if $\xi\colon \C\to X$ parametrizes (bijectively) the stable manifold of a $\mu_f$-generic point, then
$\xi^* T^-_f=a \frac{\ii}{2}dz\wedge d\bar z$ for some constant $a>0$.
\end{itemize}
These two steps occupy sections \ref{par:PESIN} to \ref{par:Renorma-Stable}.
The proof of Property (1) builds on the local product structure of $\mu_f$ and on the weak laminarity properties of $T^\pm_f$. 
The proof of (2) relies on a renormalization argument (along $\mu_f$-generic orbits).

\begin{rem}\label{latteskummer}
The renormalization techniques already appear in the proof of Latt\`es rigidity for endomorphisms of projective spaces $\P^k_\C$ (see \cite{Mayer} and \cite{Berteloot-Dupont}). Our context is actually closer to the conformal case $k=1$ since the renormalization is done along the stable and unstable manifolds. In the case of endomorphisms of $\P^k_\C$, the unstable manifolds cover open subsets of $\P^k_\C$; here, the stable and unstable manifolds have co-dimension $1$, and we have to overcome important new difficulties.
\end{rem}

\subsubsection{Normal families of entire curves} \label{norfam} We combine Zalcman's reparametri\-zation lemma, Hodge
index theorem, and a result of Dinh and Sibony to derive a compactness property, in the sense of Montel, for the entire parametrizations of  stable and unstable manifolds. This crucial step is detailed in Section~\ref{par:NormalStable}. 

\subsubsection{Laminations, foliations, and conclusion} \label{conccom} Thanks to the previous step, we prove that the stable (resp. unstable) manifolds of $f$ are organized in a (singular) lamination by holomorphic curves. Then, an argument of Ghys can be coupled to Hartogs phenomenon to show that this lamination extends to an 
$f$-invariant, singular, holomorphic foliation of $X$. 

\begin{rem}\label{latteskummer2} Thus, at this stage of the proof, the starting hypothesis on $\mu_f$ has been upgraded in a regularity property for $T^+_f$ and $T^-_f$ (these currents are smooth, and correspond to transverse invariant measures for two holomorphic foliations). For endomorphisms of  $\P^k_\C$, this fact follows directly from the renormalization argument and pluripotential theory (see Lemma 3 in \cite{Berteloot-Dupont}). 
\end{rem}

To conclude, we refer to a previous theorem of the first author and Favre concerning symmetries of foliated surfaces. 
These steps are done in Section~\ref{par:NormalStable} and Section~\ref{par:CONCLUSION}

\subsubsection{Consequences and appendices}  
Section~\ref{par:Consequences} contains the proofs of the main corollaries and consequences.
In the appendices, we state several theorems of   Dinh, Sibony,  and Moncet, and 
prove those which are not easily accessible. 

\subsection{Acknowledgements}

We thank to Eric Bedford, Romain Dujardin, S\'e\-bastien Gou\"ezel, Anatole Katok, Misha Lyubich, Fran\c{c}ois Maucourant, and Nessim Sibony for useful discussions related to this article.

\section{Invariant currents and periodic curves}\label{par:prelim}

In this section, we collect general results concerning the dynamics of automorphisms of compact K\"ahler surfaces $X$. 
We refer to \cite{Cantat:Milnor}, and the references therein, for a complete exposition. (the proofs of a few technical
complements are given in the appendix, \S~\ref{par:appendixwhole})

\subsection{Cohomology groups}

Let $X$ be a compact K\"ahler surface. 
Let $H^k(X ; \R)$ and $H^k(X ; \C)$  denote the real and complex de Rham cohomology groups of $X$. 
If $\eta$ is a closed differential form (or a closed current, see below),  its cohomology class is denoted by $[\eta]$.

For $0\leq p,\, q \leq 2$, let $H^{p,q}(X;\C)$ denote the subspace of $H^{p+q}(X;\C)$ of all classes represented by closed $(p,q)$-forms, and let $h^{p,q}(X)$ be its dimension. Hodge theory implies that 
\[
H^k(X ; \C) = \bigoplus_{p+q=k} H^{p,q}(X).
\] 
Complex conjugation exchanges $H^{p,q}(X;\C)$ and $H^{q,p}(X;\C)$; thus, $H^{1,1}(X;\C)$ inherits a real structure, with real part  
\[
H^{1,1}(X ; \R) := H^{1,1}(X;\C) \cap H^2(X;\R).
\]
The intersection form is an integral quadratic form on $H^2(X;\Z)$. It satisfies 
\[
 \forall u ,v \in H^{1,1}(X;\R) \ , \  \langle u \vert v \rangle : = \int_X \tilde u \wedge \tilde v , 
 \]  
where $\tilde u$ and $\tilde v$ are closed $2$-forms on $X$ representing $u$ and $v$. By Hodge index theorem $\langle \cdot \vert \cdot \rangle$ is non-degenerate and of signature $(1, h^{1,1}(X)-1)$  on $H^{1,1}(X;\R)$. 
This endows $H^{1,1}(X;\R)$ with the structure of a Minkowski space. 

The subset of $H^{1,1}(X;\R)$ consisting of classes of K\"ahler forms is called the {\bf{K\"ahler cone}} of $X$.  By definition, the closure of the K\"ahler cone is the {\bf{nef cone}}. These cones intersect only one of the two connected components of $\{ u \in H^{1,1}(X;\R) \ , \   \langle u \vert u \rangle = 1 \}$.
We denote by $\H(X)$ this connected component. With the distance given by 
$\cosh(\dist(u,v))=\langle u \vert v\rangle$, 
 this is a model of the hyperbolic space of dimension $h^{1,1}(X)-1$.

\subsection{Action on cohomology groups (see \cite{Cantat:Milnor}, \S 2)}\label{actioncoho}

Every $f \in \Aut(X)$ induces a linear invertible mapping $f^*$ on $H^{1,1}(X;\R)$ which is an isometry for the intersection product.  Since the K\"ahler cone is $f^*$-invariant, so is $\H(X)$. 
Let $\lambda_f$ denote the spectral radius of $f^*\colon H^{1,1}(X;\R)\to H^{1,1}(X;\R)$. 
\begin{itemize}
\item $\lambda_f$ coincides with the spectral radius of $f^*$ acting on the full cohomology space $\oplus_{k=0}^4 H^k(X;\C)$.
\item The topological entropy of $f$ is equal to  $\log \lambda_f $. 
\item When $\lambda_f> 1$  the eigenvalues of $f^*$ on $H^{1,1}(X;\R)$ (resp. on $H^2(X;\R)$) are precisely $\lambda_f$, $\lambda_f^{-1}$ (which are both simple), and complex numbers of modulus $1$. The eigenlines corresponding to $\lambda_f$ and
$\lambda_f^{-1}$ are isotropic, and they intersect the nef cone. 
\end{itemize}

From the third item, we can fix a nef eigenvector $\theta_f^+$ (resp. $\theta_f^-$) for the eigenvalue $\lambda_f$ (resp. $\lambda_f^{-1}$). Let $\Pi_f$ be the subspace generated by $\theta_f^+$ and $\theta_f^-$, and let $\Pi_f^\perp$ be its orthogonal complement in $H^{1,1}(X;\R)$. The intersection form  has signature $(1,1)$ on $\Pi_f$ and is negative definite on $\Pi_f^\perp$.

\subsection{Contraction of periodic curves}\label{contraper}

If $C\subset X$ is a complex curve, one denotes by $[C]\in H^2(X;\Z)$ the dual of the homology class of $C$ for
the natural pairing between $H^2(X;\Z)$ and $H_2(X;\Z)$. We call it the cohomology class of $C$, this 
is an element of $H^{1,1}(X;\R)$.  If $C$ is periodic, then $[C]$ is fixed by a non trivial iterate of $f^*$ and belongs to $\Pi_f^\perp$. 

\begin{pro}[see \cite{Cantat:Milnor}, \S 4.1] \label{prop:contraction}
Let $f$ be an automorphism of a compact K\"ahler surface $X$ with positive entropy. There exists a (singular)
surface $X_0$, a birational morphism $\pi \colon X \to X_0$ and  an automorphism $f_0$ of $X_0$ such that 
\begin{enumerate}
\item $\pi\circ f=f_0\circ \pi$.
\item A curve $C\subset X$ is contracted by $\pi$ if and only if $[C] \in \Pi_f^\perp$. In particular $\pi$ contracts the periodic curves of $f$.
\item If $C$ is a connected periodic curve, then the genus of $C$ is $0$ or $1$.
\end{enumerate}
\end{pro}

Indeed, the intersection form being negative definite on $\Pi_f^\perp$, Grauert-Mumford criterium allows to contract all irreducible curves $C$ with $[C] \in \Pi_f^\perp$, and only those curves. This yields the desired morphism $\pi$. Property (3) is more difficult to establish, and is due to Castelnuovo for irreducible curves and to Diller, Jackson, and Sommese for the general case (see \S~\ref{appendix:contraction}).

\begin{rem}
The surface $X_0$ is projective (see \S~\ref{appendix:contraction}) but it may be singular; 
for instance, if $f\colon X\to X$ is a Kummer example on a K3 or rational surface, then $X_0$ is automatically singular. 
We shall make use of basic notions from complex analysis and pluri-potential theory on complex analytic spaces. For this, 
we refer to the first chapter of \cite{Demailly:1985} and chapter 14 in \cite{Chirka}. 
\end{rem}

\begin{rem}
The class $\theta^+_f+\theta^-_f\in H^{1,1}(X;\R)$ has positive self intersection and is in the nef cone. After contraction of all periodic curves by the morphism $\pi_0\colon X\to X_0$, there is no curve $E$ with $\langle (\pi_0)_*\theta^+_f+(\pi_0)_*\theta^-_f \vert [E]\rangle=0$. Hence, 
$(\pi_0)_*(\theta^+_f + \theta^-_f)$ is a K\"ahler class on $X_0$ (see \S~\ref{appendix:contraction}). 
\end{rem}

\subsection{Invariant currents and continuous potentials} \label{par:invcurrents}
We refer to \cite[Chapter 3]{GH} for an account concerning currents on complex manifolds.

Let $T$ be a closed positive current on $X$. It is locally equal to $dd^c u$ where $u$ is a plurisubharmonic function.  By definition  $u$ is a  {\bf{local potential}} of $T$, it is unique up to addition of a pluriharmonic function. Every closed positive current $T$ has a cohomology class $[T]$ in $H^{1,1}(X,\R)$. For instance, if $C\subset X$
is a complex curve and $T_C$ is the current of integration on $C$, then $[T_C]=[C]$.

If $f \in \Aut(X)$ then one can define $f^*T$ by duality: the value of $f^*T$ on a $(1,1)$-form $\omega$
is equal to the value of $T$ on $(f^{-1})^*\omega$. If $u$ is a local potential for $T$, then
 $u \circ f$ is a local potential for $f^* T$. 
 
\begin{thm}[see \cite{Dinh-Sibony:2005, M, Cantat:Milnor}]\label{curcur}
Let $f$ be an automorphism of a compact K\"ahler surface $X$ with positive entropy $\log \lambda_f$.
There exists a unique closed positive current $T_f^+$ of bidegree $(1,1)$ on $X$ such that $[T_f^+]$ coincides with the nef class $\theta_f^+$. Its  local  potentials are H\"older continuous. It satisfies $f^*T_f^+ = \lambda_f T_f^+$. 
\end{thm}

Similarly, there exists a unique closed positive current $T_f^-$ of bidegree $(1,1)$ satisfying $f^*T_f^- = \lambda_f^{-1} T_f^-$. 

\begin{rem}
In a recent preprint \cite{Dinh-Sibony:preprint2}, Dinh and Sibony strengthen Theorem~\ref{curcur} by showing that $T^+_f$ (resp. $T^-_f$)
is the unique $dd^c$-closed positive current whose cohomology class is $[T^+_f]$ (resp.  $[T^-_f]$).
\end{rem}

\begin{rem} \label{pushpush}
When contracting the periodic curves, as in Proposition~\ref{prop:contraction}, one may get a singular surface $X_0$. Theorem~\ref{thm:continuous-potentials}  shows that the image of $T^\pm_f$ by $\pi\colon X\to X_0$ is a closed positive current on $X_0$ with continuous potentials, even around the singularities. 
\end{rem}

Let $C$ be a Riemann surface and $\theta\colon C\to X$ be a non-constant holomorphic 
mapping. The  pull-back $\theta^*(T^+_f)$  is  locally defined as $dd^c (u^+\circ \theta)$
where $u^+$ is a local potential; by definition, this measure (resp. its image on $\theta(C)$) 
is called the {\bf{slice}} of $T^+_f$ by $\theta$.
The same definition applies for $\theta^*(T^-_f)$. 

\subsection{Definition and properties of $\mu_f$} \label{defprop}

In what follows, we can (and do) fix the following data: a K\"ahler form $\kappa$ on $X$ and 
eigenvectors  $\theta_f^+$ and $\theta^-_f$ with respect to the eigenvalues $\lambda_f$ and $\lambda_f^{-1}$ for the endomorphism  $f^*$ of  $H^{1,1}(X;\R)$ such that 
\begin{equation}\label{eq:normalization}
\langle \theta^+_f\vert \theta^-_f\rangle = 1, \quad \langle \theta^+_f\vert [\kappa] \rangle = \langle \theta^-_f\vert [\kappa]\rangle = 1.
\end{equation}
Then, consider the currents $T^\pm_f$ provided by Theorem~\ref{curcur}. The wedge product $T_f^+ \wedge T_f^-$ is locally defined as the $dd^c$-derivative of $u^+ dd^c u^-$, where $u^+$ and $u^-$ are local potentials for $T^+$ and $T^-$. We define  $$\mu_f := T_f^+ \wedge T_f^- .$$
This is an $f$-invariant probability measure on $X$. The dynamics of $f$ with respect to $\mu_f$ is ergodic and mixing; moreover, $\mu_f$  is the unique $f$-invariant measure with maximal entropy, and it describes the equidistribution of periodic points (see Section \ref{autoabs} and \cite{Cantat:Milnor}).

\section{Parametrizations of stable manifolds}\label{par:PESIN}

In this section we see how Pesin theory provides parametrizations of the stable and unstable manifolds 
of the automorphism $f$ by entire holomorphic curves, with a control of their derivatives. 
We provide the details of the proof of this classical fact, because the control of the derivative will be used
systematically in our forthcoming renormalization argument, and this control is hard to localize in the literature. We refer 
to \cite{Jonsson-Varolin} for a similar and more general statement in arbitrary dimension.

\subsection{Oseledets theorem and  Lyapunov exponents (see \cite{Y, Katok-Hasselblatt} and \cite{BLS1, Cantat:Milnor})}

As before let $X$ be a complex projective surface and $f$ be an automorphism of $X$ with positive entropy $\log \lambda_f$. Let $T_f^\pm$ be the invariant currents introduced in Section~\ref{par:invcurrents}. The normalization chosen in Section~\ref{defprop}, 
Equation~\eqref{eq:normalization} implies that $T^+_f$ and $T^-_f$  have mass $1$ with respect to the K\"ahler form $\kappa$, and that $\mu_f = T_f^+ \wedge T_f^-$ is an $f$-invariant probability measure on $X$.

Since $\mu_f$ has positive entropy and is ergodic, it has one negative and one positive Lyapunov exponents; we denote them by $\lambda_s$ and $\lambda_u$, with $\lambda_s < 0 < \lambda_u$ (each of these exponents has multiplicity $2$ if one views $f$ as
a diffeomorphism of the $4$-dimensional manifold $X$).

In what follows, $\epsilon$ always denotes a positive real number that satisfies $\epsilon \ll \min(\vert \lambda_s \vert , \lambda_u)$. The set $\Lambda$ will be a Borel subset of $X$ of total $\mu_f$-measure; its precise definition depends on $\epsilon$ and may change from one paragraph to another. By construction, we can (and do) assume that $\Lambda$ is  invariant: indeed, $\Lambda$  can always be replaced by $\cap_{n \in \Z} f^n(\Lambda)$. 
A measurable function $\alpha : \Lambda \to ]0,1]$ is $\epsilon$-{\bf{tempered}} if it satisfies $e^{-\epsilon} \alpha(x)  \leq \alpha (f(x)) \leq  e^\epsilon  \alpha(x)$ for every $x \in \Lambda$. 

We use the same notation $\parallel \cdot \parallel$ for the standard hermitian norm on $\C^2$ and for the hermitian norm on the tangent bundle $TX$ induced by the K\"ahler form $\kappa$. The distance on $X$ is denoted $\dist_X$; ${\rm{B}}_x(r)$ is  the ball of radius $r$ centered at~$x$.

\begin{thm}[Oseledets-Pesin]\label{OP}
Let $X$ be a complex projective surface and let $f$ be an automorphism of $X$ with positive entropy. 

There exist an $f$-invariant Borel subset $\Lambda\subset X$ with $\mu_f(\Lambda)=1$,  two $\epsilon$-tempered functions $q : \Lambda \to ]0,1]$, $\beta : \Lambda \to ]0,1]$, and a family of holomorphic mappings $(\Psi_x)_{x \in \Lambda}$ satisfying the following properties.
\begin{enumerate}
\item $\Psi_x$ is defined on the bidisk  $\disk(q(x))\times \disk(q(x))$, takes values in $X$, maps the origin to the point $x$, and is a diffeomorphism onto its image.
\item $\beta(x) \parallel z_1-z_2 \parallel  \leq \dist_X(\Psi_x(z_1),\Psi_x(z_2)) \leq \parallel z_1-z_2 \parallel$ for all pairs of points
$z_1$, and $z_2$ in the bidisk.
\item The local diffeomorphism  $f_x := \Psi_{f(x)}^{-1} \circ f \circ \Psi_x$ is well defined near the origin in $\disk(q(x))\times \disk(q(x))$, and the matrix of $D_0 f_x$ is diagonal with coefficients $a(x)$ and $b(x)$ that satisfy
\[
 \vert a(x) \vert \in e^{\lambda_u} \cdot [e^{-\epsilon} , e^{\epsilon}] , \quad \vert b(x) \vert \in e^{\lambda_s} \cdot [e^{-\epsilon} , e^{\epsilon}] . 
\]
Moreover $f_x$ is $\epsilon$-close to the linear mapping $D_0 f_x$ in the ${\mathcal{C}}^1$ topology.
\end{enumerate}
\end{thm}

The global {\bf{stable manifold}} of a point $x$ is the set $W^s(x)$ of points $x'$ such that the distance between $f^n(x)$ and
$f^n(x')$ goes to $0$ as $n$ goes to $+\infty$. The {\bf{local stable manifold}} $W^{s, loc}(x)$ is the connected component of $W^s(x) \cap \Psi_x(\disk(q(x))\times \disk(q(x)))$ that contains $x$.

In $\disk(q(x))\times \disk(q(x))$, the inverse image by $\Psi_x$ of the local stable manifold     
is a vertical graph; in other words, it can  be parametrized by a holomorphic map 
\[\gamma_x : \disk(q(x)) \to \disk(q(x))\times \disk(q(x)), \quad z \mapsto (g_x(z),z), \]  
where $g_x$ is holomorphic and satisfies $g_x(0) = 0$, $g_x'(0) = 0$, and $\Lip g_x \leq 1$. 
We have $W^s(x)=\cup_{n\geq 0}f^{-n}(W^{s,loc}(f^n(x))$.

\begin{nota}\label{newnew}
For $x$ in $\Lambda$, we set  $ \sigma_x := \Psi_x \circ \gamma_x$. This is a holomorphic parametrization of the local
stable manifold $W^{s, loc}(x)$. By construction, 
\begin{itemize}
\item $\beta(x) \leq \parallel \sigma_x'(0) \parallel \leq 1$,
\item $f (W^{s,loc}(x)) \subset W^{s,loc}(f(x))$ and $W^{s, loc}(x) \subset {\rm{B}}_x(1)$,
\item $\lim_{n \to +\infty} \dist_X( f^n(x) , f^n(y))= 0$ for every $y \in W^{s, loc}(x)$.  
\end{itemize}
\end{nota}

We denote by $F_x : \disk(q(x)) \to \disk(q(f(x)))$ the mapping that satisfies
\[
 f \circ \sigma_x = \sigma_{f(x)} \circ F_x,
 \]
and  by $M_x : \disk(q(x)) \to \disk(q(f(x)))$ the linear mapping given by  
\[
M_x(z) := m_x \cdot z, \; {\text{ with }} \quad  m_x := F_x'(0).
\]
By construction,  $\vert m_x \vert \in e^{\lambda_s} \cdot [e^{-\epsilon} , e^{\epsilon}]$ (note that $m_x$ is equal to the complex
number $b(x)$ of Theorem~\ref{OP}).

\subsection{Parametrization of stable manifolds}\label{sub:parastable}

\subsubsection{Linearization along local stable manifolds}\label{par:varolin} 
The following proposi\-tion provides a linearization of $f$ along the stable manifolds. 
We keep the same notations as in the previous paragraph; in particular, $\epsilon>0$ 
is fixed and $\Lambda$ is given by Theorem~\ref{OP}.

\begin{pro}\label{xi}
Let $f$ be an automorphism of a complex projective surface $X$ with positive entropy. If $\epsilon$ is small enough, 
there exist a real number $c := c(\epsilon, \lambda_s) \in ]0,1]$ and holomorphic injective functions $(\eta_x)_{x \in \Lambda}$, such that
\begin{enumerate}
\item $\eta_x$ is defined on $\disk(c q(x))$, with values in $\disk(q(x))$, and it satisfies $\eta_x(0)=0$ and $\eta_x'(0)=1$;
\item $f \circ (\sigma_x \circ \eta_x) = (\sigma_{f(x)} \circ \eta_{f(x)}) \circ M_x$ on $\disk(c q(x))$. 
\end{enumerate}
\end{pro} 


\begin{proof} We split the proof into four steps. 

\vspace{0.16cm}

\noindent$\bullet${\sl{ Preliminary choices.-- }} Choose $\epsilon > 0$ such that $\lambda_s +6 \epsilon < 0$. Thus, the series 
\[
s(\epsilon) := \sum_{m \geq 0} e^{m(\lambda_s + 6\epsilon)}
\]
converges. Then,  decrease $\epsilon$ so that 
\[
s(\epsilon) e^{2 \epsilon} (e^\epsilon -1) \leq 1/2, \; {\text{ and }}¬†\quad e^{\lambda_s} (e^\epsilon - 1) \leq 1,
\]
and set $c:={1 \over 24} e^{\lambda_s}(e^\epsilon -1)$. Then apply Theorem~\ref{OP} and define
\[
s := s(\epsilon) \ , \   q'(x) := 12c q(x) \ {\text{ and }} \quad 
Q(x) := 2 e^{ - \lambda_s + 2 \epsilon} q(x)^{-1}.
\] 

\noindent$\bullet${\sl{ Contraction properties of $F_x$.-- }} For $x$ in $\Lambda$, denote the Taylor expansion of $F_x(z)$ by
$M_x (z) + \sum_{p \geq 2} a_p(x) z^p$. Since $F_x$ maps $\disk(q(x))$ into $\disk(q(f(x))$, Cauchy formula yelds $\vert a_p(x) \vert \leq q(f(x))/q(x)^p$ ; hence
\begin{equation}\label{fstep}
 \forall z \in \disk(  q(x) /2 ), \  \vert F_x(z) - M_x(z) \vert \, \leq \, 2 {q(f(x)) \over q(x)^2} \vert z \vert ^2  \, \leq \,  2 {e^\epsilon \over q(x)} \vert z \vert ^2 .
\end{equation}
Here, and in what follows, we use that $q$ is $\epsilon$-tempered, i.e.  $q(f(x)) \in q(x) \cdot [e^{-\epsilon} , e^\epsilon]$.
Thus, the definition of $q'(x)$ implies that 
\begin{equation}\label{sstep}
\forall z \in \disk(q'(x)), \  \vert F_x(z)  \vert \leq  \vert M_x(z) \vert + \vert F_x(z) - M_x(z) \vert \leq  e^{\lambda_s + 2 \epsilon} \vert z \vert .
\end{equation}
In particular, we obtain the inclusions 
\[
F_x(\disk(q'(x))) \subset \disk(e^{\lambda_s + 2 \epsilon} q'(x)) \subset \disk(q'(f(x))).
\]

\noindent$\bullet${\sl{The functions $\tilde \eta_x$.-- }} From the previous inclusion, we can define (for all $x$ in $\Lambda$, $z$ in $\disk(q'(x))$, and $m\geq 0$)
\[
 \tilde \eta_{m,x}(z) := M_x^{-1} \circ \cdots M_{f^{m-1}(x)}^{-1} \circ F_{f^{m-1}(x)} \circ \cdots \circ F_x (z), 
\]
with the convention $\tilde \eta_{0,x} := \Id_{\disk(q'(x))}$. These functions satisfy $\tilde \eta_{m, f(x)} \circ F_x = M_x \circ \tilde \eta_{m,x}$ on $\disk(q'(x))$. Let us show by induction that
\begin{equation}\label{intion}
\forall z \in \disk(q'(x)), \ \vert \tilde \eta_{m+1, x}(z) - \tilde \eta_{m, x}(z)  \vert \leq e^{m(\lambda_s + 6 \epsilon)} \vert z \vert ^2 Q(x) . 
 \end{equation}
The case $m= 0$ is deduced from (\ref{fstep}); indeed $\vert \tilde \eta_{1, x}(z) - \tilde \eta_{0, x}(z)  \vert$ is equal to 
\[
\vert M_x^{-1} \circ F_x (z) - z \vert = m_x ^{-1} \vert F_x(z) - M_x(z) \vert \leq 2  {e^{-\lambda_s + 2\epsilon} \over q(x) }   \vert z \vert ^2 =  \vert z \vert ^2 Q(x) . 
\]
Now assume that (\ref{intion}) is satisfied for some $m \geq 0$. In particular
\[
\forall w \in \disk(q'(f(x))), \ \vert \tilde \eta_{m+1, f(x)}(w) - \tilde \eta_{m, f(x)}(w)  \vert \leq e^{m(\lambda_s + 6 \epsilon)} \vert w \vert ^2 Q(f(x)) . 
\]
For $z\in \disk(q'(x))$, we know that $ F_x(z) \in\disk(q'(x))$, hence we can replace $w$ by $F_x(z)$ to obtain 
\[  \vert \tilde \eta_{m+1, f(x)}(F_x(z)) - \tilde \eta_{m, f(x)}(F_x(z))  \vert \leq e^{m(\lambda_s + 6 \epsilon)} \vert F_x(z) \vert ^2 Q(f(x)).
\]
Multiplying by $\vert m_x ^{-1} \vert$ we obtain
\[
 \forall z \in \disk(q'(x)), \ \vert \tilde \eta_{m+2,x}(z) - \tilde \eta_{m+1, x}(z)  \vert \leq e^{m(\lambda_s + 6 \epsilon)} \vert m_x ^{-1} \vert \vert F_x(z) \vert ^2 Q(f(x)) . 
 \]
Using (\ref{sstep}) the right hand side is less than or equal to 
\[
 e^{m(\lambda_s + 6 \epsilon)}  e^{-\lambda_s + \epsilon} e^{2(\lambda_s + 2 \epsilon)} \vert z \vert ^2 e^\epsilon Q(x) = e^{(m+1)(\lambda_s + 6 \epsilon)}\vert z \vert ^2 Q(x) , 
 \]
which establishes (\ref{intion}) for $m+1$ as desired. 

Since $\tilde \eta_{0,x}(z) = z$ we can write 
\[
\forall x \in \Lambda, \; \forall z \in \disk(q'(x)), \quad \tilde \eta_{n,x}(z) = z + \sum_{m=0}^{n-1} ( \tilde \eta_{m+1,x} - \tilde \eta_{m,x}) (z),
\]
and Equation (\ref{intion}) implies that $\tilde \eta_{n,x}$ converges locally uniformly on $\disk(q'(x))$ to a function $\tilde \eta_x$ satisfying 
\[
\tilde \eta_x'(0) = 1, \quad \vert (\tilde \eta_x  - \Id)(z) \vert \leq s  \vert z \vert ^2 Q(x), \; {\text{and}}\quad  \tilde \eta_{f(x)} \circ F_x = M_x \circ  \tilde \eta_x.
\] 

\noindent$\bullet${\sl{ Conclusion.-- }}
We now invert $\tilde \eta_x$ and $\tilde \eta_{f(x)}$. First let us verify that $\Lip (\tilde \eta_x - \Id) \leq 1/2$. For that purpose, apply Cauchy's formula 
\[(\tilde \eta_x - \Id)' (z) = \frac{1}{2i\pi} \int_C \frac{(\tilde \eta_x - \Id) (a)}{ (a-z)^{2}}\,  d\! a
\] 
where $C$ is the circle of radius $q'(x)$ and $z$ is taken in the disk $\disk(q'(x)/2)$; the previous estimate $\vert (\tilde \eta_x - \Id)(z) \vert \leq s  \vert z \vert ^2 Q(x)$ and the definitions of $q'(x)$ and $Q(x)$ yield
\[
  \vert (\tilde \eta_x - \Id)' (z) \vert \leq s Q(x) q'(x) \leq 1/2 , 
 \]
 for all points $ x \in \Lambda$, and $ z \in \disk(q'(x)/2)$; this
 implies $\Lip (\tilde \eta_x - \Id) \leq 1/2$ on $\disk(q'(x)/2)$ as desired. 
 
 We deduce 
\begin{equation}\label{opop}
     {1 \over 2} \vert z_1 - z_2 \vert \leq \vert \tilde \eta_x(z_1) - \tilde \eta_x(z_2)  \vert \leq  {3 \over 2} \vert z_1 - z_2 \vert . 
 \end{equation}
 for all points $z_1$ and $ z_2$ in $ \disk(q'(x)/2)$.
It follows that $\tilde \eta_x$ can be inverted on the disk $\disk(q'(x)/4)$, and we define  
\[
\eta_x := (\tilde \eta_x)^{-1} : \disk(q'(x)/4) \to \disk(q'(x)/2).
\] 
This function satisfies 
\begin{equation}\label{opop2}
 \forall z_1, \, z_2 \in \disk(q'(x)/4), \quad  {2 \over 3} \vert z_1 - z_2 \vert \leq \vert \eta_x(z_1) - \eta_x(z_2)  \vert \leq  {2} \vert z_1 - z_2 \vert . 
 \end{equation}
 
Compose each term of the equality  $\tilde \eta_{f(x)} \circ F_x = M_x \circ  \tilde \eta_x$ on the right by $\eta_x$; one obtains $\tilde \eta_{f(x)} \circ F_x \circ \eta_x = M_x$ on $\disk(q'(x)/4)$. Now, Equations \eqref{sstep} and \eqref{opop} lead to 
\[
\vert \tilde \eta_{f(x)} \circ  F_x \circ \eta_x  \vert \leq q'(f(x)) / 4
\] 
on $\disk(q'(x)/12)$. Composition on the left by $\eta_{f(x)}$  gives $F_x \circ \eta_x = \eta_{f(x)} \circ M_x$ on $\disk(q'(x)/12) = \disk(cq(x))$.  To complete the proof, it suffices to compose  the latter expression on the left by $\sigma_{f(x)}$; the relation $\sigma_{f(x)} \circ F_x = f \circ \sigma_x$ provides $f \circ (\sigma_x \circ \eta_x) = (\sigma_{f(x)} \circ \eta_{f(x)}) \circ M_x$ on $\disk(cq(x)))$ as desired. 
 \end{proof}

\subsubsection{Local and global stable manifolds}\label{par:affine-param} Definition \ref{newnew} and Proposition~\ref{xi} allow to introduce the following definition.

\begin{nota}
For every $x \in \Lambda$, define $\xi^{s, loc}_{x}\colon  \disk(c q(x)) \to W^{s, loc}(x)$ by 
\[
\xi^{s, loc}_{x} := \sigma_x \circ \eta_x .
\]
By construction, $\xi^{s,loc}_x$ is holomorphic, injective and satisfies 
\begin{itemize}
\item $\xi^{s, loc}_{x}(0) =x$ and  $\beta(x) \leq \vert  (\xi^{s, loc}_{x})'(0)  \vert \leq 1$; 
\item  $2\beta(x)/3 \leq \vert  (\xi^{s, loc}_{x})'   \vert \leq 2$ on $\disk(cq(x))$; 
\item $f \circ \xi^{s, loc}_{x} = \xi^{s, loc}_{f(x)} \circ M_x$ on $\disk(c q(x))$. 
\end{itemize}
\end{nota}

The global stable manifold satisfies
\[ W^s(x) = \cup_{n \geq 0} f^{-n} \, (W^{s, loc}(f^n(x))). \] 
Thus, by construction, $W^s(x)$ is a simply connected Riemann surface.  

\begin{pro}\label{affpara}
For every $x \in \Lambda$, the Riemann surface $W^s(x)$ is biholomorphic to $\C$. Moreover, there exists a biholomorphism $\xi^s_x : \C \to W^s(x)$
 such that 
 \begin{itemize}
\item$\xi^s_x(0)=x$,
\item $\xi^s_x = \xi^{s, loc}_{x}$ on $\disk(c q(x))$, 
\item $f \circ \xi^s_x  = \xi^s_{f(x)}  \circ M_x$ on $\C$. 
\end{itemize}
\end{pro}

\begin{proof}
Set $M^m_x := M_{f^{m-1}(x)} \circ \ldots \circ M_x : \C \to \C$ for every $m \geq 1$, and observe that $\vert M^m_x(z) \vert \in e^{m\lambda_s} \cdot [e^{-m\epsilon} , e^{m\epsilon}] \cdot \vert z \vert$. Then, define 
\[
\forall z \in \C, \ \ \xi^s_x (z) : = f^{-m(z)} \circ \xi_{f^{m(z)}(x)}^{s, loc} \circ M^{m(z)}_x(z) ,
\]
where $m(z)$ is a large positive integer, so that $M^{m(z)}_x(z) \in \disk(c q(f^{m(z)}(x)))$; such an integer exists because the function $q$ is $\epsilon$-tempered. One easily verifies that (i) the definition of $\xi^s_x$ does not depend on $m(z)$ and (ii) $f \circ \xi^s_x  = \xi^s_{f(x)}  \circ M_x$ on $\C$  by analytic continuation. The map $\xi_x^s :  \C \to W^s(x)$ is a biholomorphism by the definition of $W^s(x)$ and the fact that $f$ has empty critical set.
\end{proof}

\begin{rem} \label{rem:affpara2}
The definition of the biholomorphism  $\xi_x^s : \C \to W^s(x)$ depends on the local parametrizations $(\xi^{s, loc}_{x})_{x \in \Lambda}$, hence on Pesin's theory. In particular, the derivative $(\xi_x^s)'(0)$ depends measurably on $x$. 

Two biholomorphisms $\C \to W^s(x)$ differ by an affine automorphism $z \mapsto  a z + b$ where $(a,b) \in \C^*\times \C$. Thus, (i) the stable manifold $W^s(x)$ inherits a natural affine structure and (ii) every biholomorphism $\C \to W^s(x)$ sending $0$ to $x$ is equal to $\xi_x^s$ modulo composition with a homothety $z\mapsto a z$. 
\end{rem}

\begin{rem} \label{rem:affpara3}
The fact that the stable manifolds are isomorphic to $\C$ is well known. Let us summarize the proof given in \cite[section 2.6]{BLS1}. First, remark that the stable manifolds are Riemann surfaces which are homeomorphic to $\R^2$; thus, $W^s(x)$ is conformally equivalent to the unit disk $\disk$ or the affine line $\C$. By ergodicity of $\mu_f$, one can assume that almost all stable manifolds have the same conformal type. If they were disks, $f$ would induce isometries from $W^s(x)$ to $W^s(f(x))$ with respect to the Poincar\'e metric, because   conformal automorphisms 
of the unit disk are isometries; this would contradict the fact that $f$ is a contraction along the stable manifolds. 
\end{rem}

\subsection{Ahlfors-Nevanlinna currents and stable manifolds}

From Proposition~\ref{affpara},  the stable manifolds of $f$ are parametrized by entire curves $\xi^s_x\colon \C\to W^s(x)$. We associate Ahlfors currents to them: these closed positive currents coincide with $T^+_f$.

\subsubsection{Ahlfors-Nevanlinna currents}\label{sub:ahlfors}

Let $M$ be a complex space, let $\kappa$ be a k\"ahler form on $M$, and let $\xi \colon \C \to M$ be a non-constant entire curve.  Let $A(r;\xi)$ be the area of $\xi(\disk_r)$, defined as the integral of $\xi^*\kappa$ on $\disk_r$, and let $T(r;\xi)$ be the logarithmic average 
\[
T(r;\xi) : =\int_{s=0}^{r}A(s;\xi)\frac{ds}{s}.
\]
The family of Ahlfors-Nevalinna currents $N(r;\xi)$ is then defined by
\[
N(r;\xi)\colon \alpha \mapsto \frac{1}{T(r;\xi)}\int_{s=0}^r \{\xi(\disk_s) \} (\alpha) \frac{ds}{s}
\]
where $\{\xi(\disk_s)\}$ is the current of integration on $\xi(\disk_s)$ (counting multiplicities).

\begin{lem}[Ahlfors, see \cite{Brunella:1999}]
There exist sequences $(r_i)_{i\in \N}$ tending to $+\infty$ such that $N(r_i;\xi)$ converges to a closed positive current on $X$. 
\end{lem}

By definition, these closed currents are the {\bf{Ahlfors-Nevanlinna currents}} determined
by $\xi$. Every such current $A$ has a cohomology class $[A]$ in $H^{1,1}(X,\R)$ (see \S~\ref{par:invcurrents}). 

\begin{pro}[see \cite{Brunella:1999}]\label{geneahlfors}
Let $X$ be a compact K\"ahler surface and $\xi\colon \C\to X$ be a non-constant entire curve. 
Let $A$ be an Ahlfors current determined by $\xi$.
\begin{itemize}
\item[(1)] If $\xi(\C)$ is contained in a curve $E$, then the genus of $E$ is equal to $0$ or $1$ and $A$ is
equal to $Area(E)^{-1}\{E\}$. 
\item[(2)] If the area $A(r;\xi)$ is bounded by a constant which does not depend on $r$, then $\xi(\C)$
is contained in a compact curve $E\subset X$. 
\end{itemize}
If $\xi(\C)$ is not contained in a compact curve, then 
\begin{itemize}
\item[(3)]  $\langle[A]\vert [C]\rangle \geq 0$ for every curve $C \subset X$;
\item[(4)] $[A]$ is in the nef cone of $X$ and $[A]\cdot [A]\geq 0$.
\end{itemize}
\end{pro}

\begin{rem}
In Section~\ref{par:ds-corps-texte} and Section~\ref{par:ds-appendix}, we shall also study the {\bf{Ahlfors currents}}, defined as closed positive limits of 
\[
\frac{1}{A(r;\xi)} \{\xi(\disk_r\}.
\]
Such closed currents exist, but Properties (3) and (4) may a priori fail for them. 
\end{rem}

\subsubsection{Stable and unstable Ahlfors currents}
By the following theorem, the Ahlfors currents associated to  the stable manifolds of $f$ are equal to $T^+_f$; 
to state it, recall that $\Lambda$ is the set of full $\mu_f$-measure introduced in Theorem~\ref{OP}.

\begin{thm}[see \cite{Cantat:Milnor}]\label{TequalA} Let $f$ be an automorphism of a compact K\"ahler
surface $X$ with positive entropy. Let  $x$ be a point of $\Lambda$, and $\xi^s_x\colon \C\to X$ be a  parametrization of the stable manifold $W^s(x)$. 
If $\xi^s_x(\C)$ is not contained in a (compact) periodic curve, then all Ahlfors-Nevanlinna currents associated to $\xi^s_x$ coincide
with $T^+_f$.
\end{thm}

For $\xi^s_x$ one can take the parametrization of Proposition~\ref{affpara}. A similar result holds for unstable manifolds and the current $T^-_f$. 

\begin{rem} 
In \cite{Dinh-Sibony:preprint2}, Dinh and Sibony prove the following strengthening of Theorem~\ref{TequalA}: {\sl{if $\xi\colon \C\to X$ is 
an entire curve such that 
\begin{itemize}
\item $\xi(\C)$ is not contained in an algebraic periodic curve of $f$, and
\item the family of entire curves $f^n\circ \xi$, $n\geq 1$, is locally equicontinuous (i.e. is a normal family of
entire curves),
\end{itemize}
then all Ahlfors-Nevanlinna currents of $\xi$ coincide with $T^+_f$}}. We shall not need this version of Theorem~\ref{TequalA}.
\end{rem}

\subsubsection{Dinh-Sibony theorem}\label{par:ds-corps-texte}

The following result is proved for complex manifolds in \cite{Dinh-Sibony:preprint}. The proof 
extends to the case of compact complex spaces. To state it, recall that a {\bf{Brody curve}} is a non-constant entire
holomorphic curve $\xi\colon \C\to M$ whose velocity $\parallel \xi'(z)\parallel$ is uniformly bounded. 

\begin{thm}[Dinh-Sibony, \cite{Dinh-Sibony:2005}]\label{thm:DS} Let $M$ be a compact complex analytic space. Let $T$ be a closed 
positive current of bidegree $(1,1)$ on $M$ which is locally defined by continuous potentials. Let $\xi \colon \C \to M$
be a Brody curve such that $\xi^*(T)=0$. Then, there is an Ahlfors current $S$ determined by $\xi$ that satisfies $T\wedge S=0$.  
\end{thm}

\begin{rem}
Note that $\xi^*T$ and $T\wedge S$ are well defined because $T$ has continuous potentials: locally, 
$\xi^*T$ is the positive measure defined by $dd^c(u\circ\xi)$, where $u$ is 
a local potential for $T$; similarly,  $T\wedge S$ is locally defined as the current $dd^c(u\cdot S)$.
\end{rem} 

The following corollary is a crucial ingredient in the proof of Theorem \ref{thm:normal}; it is  proved in Section~\ref{app-appendix}, 
together with  a corollary concerning Fatou sets of automorphisms of complex projective surfaces.

\begin{cor}\label{cor:DS}
Let $X$ be a complex projective surface and let $f$ be an automorphism of $X$ with positive entropy. 
Let $\pi\colon X\to X_0$ be the contraction of the periodic curves of $f$. Then, \begin{itemize}
\item[(1)] the Ahlfors currents $A_\nu$ of every non-constant entire curve $\nu\colon \C\to X_0$  satisfy 
\[
 \int_X A_\nu\wedge \pi_*(T^+_f+T^-_f) \, >0;
\] 
(this product is well defined because $\pi_*T^+_f$ and $\pi_*T^-_f$ have continuous potentials on $X_0$)
\item[(2)] there is no non-constant entire curve $\xi\colon \C\to X_0$ with $\xi^*(\pi_*(T^+_f+T^-_f))=0$.
\end{itemize}
\end{cor}

\section{Product structure and absolute continuity}\label{par:abso-conti}

The currents $T^{\pm}_f$ have a common geometric property, called weak laminarity. To describe this property, we relate
it to the dynamical notion of Pesin boxes, and explain that the measure $\mu_f$ has a product structure in these
boxes. These properties lead to the proof of Proposition~\ref{pro:abso-conti}: it says that  
all slices of $T^-_f$ and $T^+_f$ by Riemann surfaces give rise to absolutely continuous 
measures; in other words, one can transfer the regularity assumption on $\mu_f$ to a (rather weak) regularity property 
of $T^+_f$ and $T^-_f$. 

We refer to \cite{BLS1, Cantat:Acta, Dujardin:Duke}, as well as \cite{Cantat:Milnor}, for proofs of the results used in this section. 

\subsection{Laminations and quasi-conformal homeomorphisms}

\subsubsection{Quasi-conformal homeomorphisms (see \cite{Ahlfors:book})}\label{quasico}

Let $h\colon \U \to \V$ be an orientation preserving homeomorphism between two Riemann surfaces. 
One says that $h$ is $K$-quasi-conformal if $h$ is absolutely continuous on lines
and 
\[
\vert \partial_{\overline{z}} h \vert \leq \frac{K-1}{K+1} \vert \partial_{z} h \vert
\]
almost everywhere, see \cite[Chapter II]{Ahlfors:book}. A $1$-quasi-conformal 
mapping is conformal, hence holomorphic.

\begin{lem}[see \cite{Ahlfors:book},Chapter II, Theorem 3]
If $h$ is a quasi-conformal homeomorphism, then $h$ is absolutely continuous 
with respect to Lebesgue measure. 
\end{lem}

To describe this statement in more details, fix a local co-ordinate $z$ on $\U$,
with Lebesgue measure $\Leb$ equal to the area form 
\[
 \frac{\ii}{2}dz\wedge d\zbar.
\]
 Then, the partial derivatives of $h$ 
are well defined almost everywhere because $h$ is absolutely continuous on lines,
its jacobian determinant is locally integrable, and
\[
\Leb(h(A))=\int_A {\sf{Jac}}(h)(z)\frac{\ii}{2}dz\wedge d\zbar
\]
for every Borel subset $A$ of $\U$. 

\subsubsection{Laminations in bidisks (see \cite{Douady:Bourbaki, Ghys:Lamination})}\label{par:lamination-intro}

By definition, a {\bf{horizontal graph}} in the bidisk $\disk\times \disk$ is the graph $\{(z,\varphi(z)); z\in \disk\}$ of a 
holomorphic function $\varphi\colon \disk \to \disk$; thus, horizontal graphs are smooth analytic subsets of 
$\disk\times \disk$ that intersect every vertical disk $\{z\} \times \disk$ in exactly one point. Vertical graphs are
images of horizontal graphs by permutation of the co-ordinates.

Consider a family of disjoint, horizontal, holomorphic graphs in  $\disk\times \disk$. If $m$ is
a point on the transversal $\{0 \}\times \disk$ which is contained in one of these graphs, 
one denotes by $\varphi_m\colon \disk \to \disk$ the holomorphic function such that 
$z\mapsto (z,\varphi_m(z))$ parametrizes the graph through $m$. By Montel and Hurwitz theorems, 
one can extend this family of graphs in a unique way into a lamination $\L$ of a compact
subset $\K$ of $\disk\times \disk$ by disjoint horizontal graphs. The leaf of $\L$ through a point
$m$ is denoted $\L(m)$.  

If $z_1$ and $z_2$ are two points on $\disk$, the vertical disks 
\[
\Delta_j=\{(z,w)\, \vert \; z=z_j\}\quad (j=1, 2)
\]
are transverse to the lamination $\L$. 
Denote by $h_{z_1,z_2}$ the holonomy from $\Delta_1$ to $\Delta_2$. 
More generally, if $\Delta$ and $\Delta'$ are two complex analytic transversals (intersecting each
leaf into exactly one point), one gets a holonomy map from $\Delta$ to $\Delta'$. 
By the $\Lambda$-Lemma (see \cite{Douady:Bourbaki}), {\sl{the holonomy  is automatically quasi-conformal}}; 
in particular, it is absolutely continuous with respect to Lebesgue measure.  The quasi-conformal constant of $h_{z_1,z_2}$ satisfies 
\[
0\leq K(z_1,z_2)-1\leq \vert z_2-z_1\vert.
\]
Hence it converges to $1$ 
when $\Delta'$ converges to $\Delta$ in the ${\mathcal{C}}^1$-topology.

\subsection{Laminarity and Pesin boxes}\label{pesinboxes}

\subsubsection{Pesin boxes (see \cite[Section 4]{BLS1}, and \cite{Cantat:Acta, Cantat:Milnor, Dujardin:Duke})}\label{par:pesinboxesdef}

A {\bf{Pesin box}} $\Pes$ for the automorphism $f\colon X\to X$ consists
in an open subset $\U$ of $X$ which is biholomorphic to a bidisk $\disk\times \disk$ 
together with two transverse laminations $\L^s$ and $\L^u$. To fix the ideas the lamination 
$\L^u$ is horizontal: its leaves $\L^u(m)$ are horizontal graphs. These graphs
$\L^u(m)$ intersect the vertical transversal $\{0\}\times \disk$ onto a compact
set $\K^-$ and the union of these graphs is homeomorphic to the product $\disk\times \K^{-}$.
Similarly, $\L^s$ is a lamination by vertical graphs with support  homeomorphic
to $\K^+\times \disk$. 

 Given a point $w\in \K^-$ and a point $w'\in \K^+$, the horizontal leaf $\L^u((0,w))$ intersect the 
vertical leaf $\L^s((w',0))$ in a unique point $[w,w']\in \U$. This provides a homeomorphism $h$
between the product $\K^-\times \K^+$ and the intersection $\K\subset \U$ of the 
supports of $\L^s$ and $\L^u$.  Moreover, by definition, a Pesin box $\Pes=(\U, \L^u, \L^s)$ 
must satisfy the following properties.
\begin{enumerate}
\item[(0)] For $\mu_f$-almost
every point $x\in \K$, the leaf $\L^u(x)$ (resp. $\L^s(x)$) is contained in the global stable
manifold $W^u(x)$ (resp. $W^s(x)$).
\item[(1)] There is a measure $\nu^+$ whose support is $\K^+$  such that the laminar current 
\[
T^+_\Pes:=\int_{w\in \K^+} \{\L^s(w)\} d\nu^+(w)
\]
is dominated by the restriction of $T^+_f$ to $\U$ and 
coincides with $T^+_f$ on the set of continuous $(1,1)$-forms whose support
is a compact subset of the support of $\L^s$. 

\item[(2)] There is, similarly, a uniformly laminar current $T^-_\Pes$ associated to the lamination $\L^u$ and
a transverse measure $\nu^-$ whose support is $\K^-$; this current is the restriction of $T^-_f$ to the support of the unstable lamination~$\L^u$.

\item[(3)] Via the homeomorphism $h\colon \K^-\times \K^+\to \K$, the measure $\mu_f$ corresponds to the product
measure $\nu^+\otimes \nu^-$, i.e.  $\mu_{f\vert \K} = h_*(\nu^+\otimes \nu^-)$.

\end{enumerate}

In a Pesin box $\Pes$, the measure $\nu^+$ can be identified to the conditional measure
of $\mu_f$ with respect to the lamination $\L^s$ (see property (3)). One way to specify this fact is the following. By property (1), 
one can slice $T^+_f$ with an  unstable 
leaf $\L^u(m)$ to get a measure $(T^+_f)_{\vert \L^u(m)}$ (see \S~\ref{par:invcurrents}); then restrict this measure to the intersection of $\L^u(m)$
with the support of the stable lamination, and then push it on $\K^+$ (using the holonomy of $\L^s$); again, one
gets $\nu^+$.

{\sl{Pesin boxes exist, and their union has full $\mu_f$-measure}}: this comes from Pesin theory of non-uniformly 
hyperbolic dynamical systems, and from the fact that $T^+_f$ and $T^-_f$ are  Ahlfors currents of entire curves parametrizing generic
stable and unstable manifolds. See \cite[Section 4]{BLS1} (and also \cite{Cantat:Acta, Cantat:Milnor, Dujardin:Duke}).

\subsubsection{Laminar structure of $T^\pm_f$ (see \cite{BLS1})}\label{par:T+laminar}

The previous section says that $T^\pm_f$  is uniformly laminar current in each Pesin box.
In fact, $T^\pm_f$ is a sum of such currents. More precisely, there is a countable family of
Pesin boxes $\Pes_i=(\U_i, \L^u_i, \L^s_i)$, with transverse measures $\nu^\pm_i$, 
such that the support of the stable laminations
$\L^s_i$ are disjoint, and $T^+_f$ is the sum 
\[
T^+_f=\sum_i T^+_{\Pes_i}
\]
where $T^+_{\Pes_i}$ is the laminar current 
\[
T^+_{\Pes_i}=\int_{w\in \K^+} \{\L_i^s(w)\} \; d\nu_i^+(w).
\]

\subsection{Absolute continuity of the slices of the invariant currents}

\subsubsection{Absolute continuity of the transverse measures $\nu^\pm$ in Pesin boxes}\label{par:Pesin-AbsoC}

\begin{lem}
Let $\Pes=(\U, \L^u, \L^s)$ be a Pesin box with transverse measures $\nu^+$ and $\nu^-$ as in section \ref{pesinboxes}. If $\mu_f$ is absolutely continuous with respect to Lebesgue measure on $X$, then $\nu^+$ and $\nu^-$  are absolutely continuous with respect to Lebesgue measure. 
\end{lem}

The proof resides on the absolute continuity of the holonomy of $\L^u$ and $\L^s$, which we obtained from the $\Lambda$-lemma. We provide this proof because it is closely related to the arguments of Section~\ref{par:lam-to-holo}. There is a more general approach, due to Pesin, which necessitates a direct proof of the absolute continuity of the stable and unstable laminations (see \cite{Bar-Pesin}, chapter 8, and \cite{Pesin:ETHZ}, chapter 7). 
 
\begin{proof}
Let $\Delta$ be the vertical disk $\{0\}\times \disk$; it is transverse to $\L^u$. Let $A\subset \Delta$ be 
a Borel subset with Lebesgue measure $0$. Let $\L^u(A)$ be the union of the
leaves of $\L^u$ that intersect $A$. Since the holonomy maps are absolutely continuous (see \S~\ref{par:lamination-intro}), 
every  slice of $\L^u(A)$ by a vertical disk has Lebesgue measure $0$. Thus, by 
Fubini theorem and the absolute continuity of $\mu_f$, $\mu_f(\L^u(A))=0$. 
Since $\mu_f=h_*(\nu^+\otimes \nu^-)$ in $\K$ (see Property (3) of Pesin boxes), one concludes that $\nu^+(A)=0$.
This shows that $\nu^+$ is absolutely continuous with respect to Lebesgue measure. 
The argument is similar for~$\nu^-$.
\end{proof}

\subsubsection{Slices of the invariant currents}
 The following general lemma will be useful several times.

\begin{lem}\label{lem:pull-back-cv}
Let $M$ be a complex manifold. Let $T$ be a closed positive $(1,1)$-current with continuous potentials on $M$.
Let ${\U}$ be an open subset of $\C$. Let $\nu_n\colon {\U} \to M$ be a sequence of holomorphic mappings
that converges uniformly to $\nu\colon {\mathcal{U}}\to M$ on compact subsets of ${\U}$.  Then, the sequence
of measures $\nu_n^*T$ converges weakly to $\nu^*T$ as $n$ goes to $\infty$.
\end{lem}

\begin{proof}
Let $\V$ be an open subset of $M$ on which $T$ is given by a continuous potential $u$. If $\nu$
maps $\U'\subset \U$ into $\V$, then for every test function $\varphi$ with support contained in 
$\U'$, the dominated convergence theorem implies that 
\[
\langle \nu^*(T) \vert \varphi \rangle = \int_{\U'} u\circ \nu(z) dd^c\varphi(z)= \lim_{n\to \infty} \int_{\U'} u \circ \nu_n(z)dd^c\varphi(z) = \lim_{n\to \infty} \langle \nu_n^*T\vert \varphi\rangle.
\]
The result follows.\end{proof}

\begin{pro}\label{pro:abso-conti}
Let $C$ be a Riemann surface. Let $\theta\colon C\to X$ be a non-constant holomorphic 
mapping. If $\mu_f$ is absolutely continuous with respect to Lebesgue measure (on $X$), 
then the measures $\theta^*(T^+_f)$ and $\theta^*(T^-_f)$ are absolutely continuous with respect to 
Lebesgue measure (on $C$). 
\end{pro}

\begin{proof}
According to Section~\ref{par:T+laminar}, the current $T^+_f$ is a countable sum of uniformly laminar currents 
\[
T^+_f=\sum_{j=1}^\infty T^+_{\Pes_j} ,
\]
where $\Pes_j=(\U_j, \L^u_j, \L^s_j)$  is a family of Pesin boxes with disjoint supports for the laminations $\L^s$. 
Fix one of these  Pesin boxes $\Pes_j$ and an open subset $\V$ of $C$ such that $\theta_{\vert \V}\colon \V\to X$ is injective and $\theta(\V)$ is transverse to 
$\L^s_j$. It is proven in \cite[Lemmas 8.2 and 8.3]{BLS1} that $T^+_{\Pes_j}$ has continuous potentials 
and that $\theta^*(T^+_{\Pes_j})$ coincides on $\V$ with the pull back of the transverse measure $\nu^+_j$. Since
$\theta$ is holomorphic and the holonomy of $\L^s$ is absolutely continuous, the pull back 
of this measure is absolutely continuous. Thus, $\theta^*(T^+_{\Pes_j})$ is absolutely continuous
on the complement of the zeros of $\theta'$ and the tangency points between $\theta(C)$ and $\L^s$. 
The following lemma shows that the set of points $z \in C$ such that $\theta(z)$ is tangent to the lamination $\L^s_j$
is actually a discrete subset of $C$. The same property holds for the zeros of $\theta'$.

\begin{lem}[\cite{BLS1}, Lemma 6.4]\label{lem:BLS6.4} Let $C$ and $D$ be complex submanifolds of 
$\C^2$ such that (i) $C\cap D = \{p\}$ and (ii) $T_pC=T_pD$. Let $U$ be a bounded 
neighborhood of $p$. If $D'$ is sufficiently close to $D$ but $D\cap D'=\emptyset$,
then the intersection of $D'$ and $C$ in $U$ is non-empty and non-tangential 
at all intersection points. 
\end{lem}

Since $\theta^*(T^+_{\Pes_j})$ has a continuous potential, this measure has no atom. Thus, $\theta^*(T^+_{\Pes_j})$ is absolutely continuous
with respect to the Lebesgue measure on $C$. Coming back to $T^+_f$, we deduce that its pull-back is a countable sum of absolutely continuous
measures; since $\theta^*(T^+_f)$ is locally finite, this measure has a $L^1_{loc}$ density on $C$. \end{proof}

\subsection{Lebesgue density points}

Assume that $\mu_f$ is absolutely continuous with respect to Lebesgue measure. According to Proposition \ref{pro:abso-conti} the pull-back of $T^+_f$ and $T^-_f$ by a curve 
$\theta\colon C\to X$ are absolutely continuous: in local co-ordinates 
\[
\theta^*(T^\pm_f)=\varphi^\pm(z)\frac{\ii}{2}dz\wedge d\zbar 
\]
where $\varphi^+$ and $\varphi^-$ are non negative elements of $L^1_{loc}(\Leb)$. Recall that a {\bf{Lebesgue density point}}   for a function $\varphi\in L^1_{{\rm{loc}}}(\Leb)$ is a point $z$ such that
\[
\frac{1}{\pi r^2}\int_{\disk(z,r)}\vert \varphi(w)-\varphi(z)\vert \frac{\ii}{2}dw\wedge d\wbar \rightarrow 0
\]
as $r$ goes to $0$. This notion does not depend on the choice of local co-ordinates; moreover, the set
of density points of $\varphi$ has full Lebesgue measure. Thus, on each curve
$W\subset X$, there is a well defined set of density points 
\[
\Dens(W;T^+_f)=\{m\in W\, \vert \; m \; \, {\text{is a density point of }} T^+_{f\vert W}\}.
\]
In particular, each unstable manifold $W^u_f(x)$ contains a subset of full Lebesgue measure
$\Dens(W^u_f(x); T^+_f)$.

Let us now fix a Pesin box $\Pes=(\U, \L^s, \L^u)$, together with an isomorphism $\U\simeq \disk \times \disk$; 
the unstable lamination $\L^u$ is a union of horizontal graphs and the stable one is a union of vertical graphs.
One can identify $\L^u$ to a family of horizontal disks $\disk\times \{y\}$ for $y$ in $\K^-$ in two ways: 
\begin{itemize}
\item with the holonomy maps $h_{0, z}$ from the vertical $\{0\}\times \disk$ to the vertical $\{z\}\times \disk$; this parametrization is 
given by 
\[
(x,y) \in \disk \times \K^- \mapsto (x, h_{0,x}(y)).
\]
\item with the graphs $\varphi_m\colon \disk \to \disk$:
\[
(x,y)\in \disk\times \K^-\mapsto (x,\varphi_{(0,y)}(x)).
\]
\end{itemize}
We obtain a similar pair of homeomorphisms between $\K^+\times \disk$ and the support of $\L^s$. Each leaf $\L^u(m)$  contains a set of density points $\Dens(\L^u(m); T^+_f)$ and the union of these sets 
is a subset of the support of $\L^u$. 

\begin{lem} \label{lem:fullmeasure}
Let $\Pes$ be a Pesin box. If $\mu_f$ is absolutely continuous, the union of the density 
points $D := \cup_m \Dens(\L^u(m); T^+_f)$ is a subset of the support of $\L^u$ with full Lebesgue measure (resp. $\mu_f$-measure).
\end{lem}

\begin{proof}
Consider the second parametrization of the lamination $\L^u$,
\begin{eqnarray*}
 \disk\times \K^-  \longrightarrow  \disk\times \disk \simeq \U  \\
\Phi\colon  (x,y)  \mapsto  (x, \varphi_{(0,y)}(x) )
\end{eqnarray*}
and pull-back the density set $D$ by $\Phi$ in $\disk\times \disk$. Since each $\varphi_{(0,y)}$
is holomorphic, the pull-back of $\Dens(\L^u(0,y);T^+_f)$ on the horizontal disk $\disk\times¬†\{y \}$
has full Lebesgue measure. Thus, by Fubini theorem, it has full Lebesgue measure on 
$\disk\times \K^-$. Then, slice $\Phi^{-1}(D)$ by vertical disks and apply Fubini theorem again: for almost
all points $x$ in $\disk, \{x\} \times \K^-$ intersects $\Phi^{-1}(D)$ onto a subset of full measure $\{x\} \times \K^-$. 
Then, come back to $\U$ with the first parametrization: since the holonomy maps $h_{0,x}$ are absolutely 
continuous and map verticals to verticals, the result follows from Fubini theorem (with respect 
to the first projection in $\U\simeq \disk\times \disk$).
\end{proof}

\section{Renormalization along stable manifolds}\label{par:Renorma-Stable}

Our main goal in this section is the following theorem. 

\begin{thm}\label{thm:pullback-lebesgue}
Let $f$ be an automorphism of a complex projective surface $X$ with positive entropy. 
Assume that the measure of maximal entropy $\mu_f = T_f^+ \wedge T_f^-$ is absolutely continuous with respect to 
Lebesgue measure. Then there exists a measurable subset $\Lambda\subset X$ such that (i) $\mu_f(\Lambda)=1$ 
and (ii) every stable manifold $W^s(x)$ for $x\in \Lambda$ is parametrized by an injective entire curve
$\xi_x^s\colon \C\to W^s(x)$ satisfying 
\begin{equation}\label{sas}
\xi_x^s(0)=x \ \ \textrm{ and } \ \  (\xi^s_x)^*T^-_f=\frac{\ii}{2}dz\wedge d{\overline{z}}.
\end{equation}
\end{thm}

\begin{rem} The parametrization of an unstable manifold $W^s(x)$ by $\C$ is unique up to composition by an affine transformation $z\mapsto az+b$ of $\C$. Thus,  
\begin{enumerate}
\item every biholomorphism $\C \to W^s(x)$ with properties (\ref{sas}) is equal to $\xi_x^s$ up to composition by a homothety $z\mapsto az$ with $\vert a\vert =1$;
\item  the parametrization $\xi^s_x$ is the same as the parametrization defined in Section \ref{par:affine-param} up to 
composition by a dilatation $z\mapsto az$, $a\neq 0$. This is the reason why we do not introduce a new notation. 
\end{enumerate}
\end{rem}

We prove Theorem~\ref{thm:pullback-lebesgue} in three steps:
\begin{itemize}
\item (see \S~\ref{par:densitecste}) We exhibit local parametrizations $\xi_x$ of a neighborhood of $x$ in $W^s(x)$ such that
$\xi_x^*T^-_f = \alpha(x) \frac{\ii}{2}dz\wedge d{\overline{z}}$.

\item (see \S~\ref{par:densitecste-affine}) Let $\xi_x^s : \C \to W^s(x)$ be the global parametrization of $W^s(x)$ defined in section \ref{par:affine-param}. Using the first step we obtain
that  $(\xi^s_x)^*T^-_f= \alpha(x) \vert h_x(z)\vert^2 (i/2)dz\wedge d{\overline{z}}$ for some holomorphic function $h_x$ on $\disk(\beta(x)\rho(x)/4)$. 

\item Finally, we show that $\vert h_x \vert$ is indeed constant by using recurrence and exhaustion arguments.

\end{itemize}

\subsection{First step: smoothness of a local density}\label{par:densitecste}

In the following proposition $\Lambda$, $\beta$ and $q$ are respectively the measurable set and the $\epsilon$-tempered functions  introduced in Theorem~\ref{OP}. Let $c > 0$ be the constant introduced in Proposition~\ref{xi}.

\begin{pro}\label{densitecste}
Let $f$ be an automorphism of a complex projective surface $X$ with positive entropy $\log \lambda_f$. Assume that $\mu_f$ is 
absolutely continuous with respect to Lebesgue measure. 
\begin{enumerate}
\item   Then for every $x \in \Lambda$ there exist $\rho(x) > 0$ and an injective holomorphic mapping $\xi_x : \disk(\rho(x)) \to W^{s, loc}(x)$ such that 
\begin{eqnarray*}
&(i) & \xi_x(0)  =  x  \ {\textrm{ and}} \quad \beta(x)  \leq   \vert \xi_x'(0) \vert  \leq  1,  \\
&(ii) & 2\beta(x)/3  \leq   \vert \xi_x'(z) \vert  \leq  2  \ {\textrm{ on }} \  \disk(\rho(x)), \\
& (iii) & \xi_x^* T_f^-  =  \alpha(x) \cdot   \frac{\ii}{2} dz\wedge d \bar z \ {\textrm{ on }}\; \disk(\rho(x))  \ {\textrm{ for some }} \; \alpha(x) > 0.
\end{eqnarray*}
 
\item The Lyapunov exponents of $\mu_f$ satisfy
\[
\lambda_s = - {1 \over 2} \log \lambda_f \ \textrm{ and } \  \lambda_u =  {1 \over 2} \log \lambda_f . 
\]
\end{enumerate}
\end{pro}

\begin{proof} 
We prove the first assertion together with the estimate $\vert \log \lambda_f +2\lambda_s\vert \leq 2\epsilon$. The
second assertion then follows from this estimate, applied  to both $f$ and  $f^{-1}$ for arbitrary small $\epsilon >0$. 

\vspace{0.16cm}

$\bullet${\sl{ Good subsets $Q_{l,m}$.--}} To prove the proposition it suffices to work in a fixed Pesin box $\Pes$
because the union of all Pesin boxes has full $\mu_f$-measure. Let us recall that $\sigma_x : \disk(q(x)) \to W^{s, loc}(x)$ is the injective parametrization of the local stable manifold introduced in Section~\ref{par:affine-param}. 
Let $\eta_x  : \disk(c q(x)) \to \disk(q(x))$ be the holomorphic function of Proposition \ref{xi}; it satisfies 
$\eta_x(0)=0$, $\eta_x'(0)=1$, and
\[ f \circ (\sigma_x \circ \eta_x) = (\sigma_{f(x)} \circ \eta_{f(x)}) \circ M_x .
\] 
We recall that $\xi^{s, loc}_x = \sigma_x \circ \eta_x$ and that it is equal to the restriction of $\xi^s_x$ on $\disk(c q(x))$. Changing $\Lambda$ in another invariant subset of full measure if necessary, Lemma \ref{lem:fullmeasure} implies that for every $x \in \Pes \cap \Lambda$ there exists a function $\varphi_x \in L^1_{loc}(\disk(cq(x)))$ such that $0$ is a Lebesgue density point of $\varphi_x$ and 
 \begin{equation}\label{densitecourant} 
 (\xi^s_x)^*T_f^- = \varphi_x(z) \cdot \frac{\ii}{2}  dz \wedge d\bar z \ \textrm{ on } \ \disk(c q(x)) . 
 \end{equation}
 Since the origin $0$ is a Lebesgue density point of $\varphi_x$, the value $\varphi_x(0)$ is a well
 defined non-negative number. Let us define  for every $l \geq 1$:
\[
 Q_l := \Pes  \cap \Lambda \cap \{ 1/l \leq c q(x) \} \cap \{ 1/l \leq \beta(x)  \leq 1 \} \cap \{ 1/l \leq \varphi_x(0) \leq l   \} .
\]
Then apply Lusin's theorem to find for every $m \geq 1$ a subset $Q_{l,m}\subset Q_l$ of measure $(1-1/m)\mu_f(Q_{l})$ on which $\beta$ is continuous. One may assume $Q_{l,m}\subset Q_{l,m+1}$, and we have  
\[ \mu_f(\cup_{l,m \geq 1} Q_{l,m})  =  \mu(\Pes) . \]
 Fix a pair of integers $(l,m)$ and denote $Q_{l,m}$ by $Q$ in what follows. Since the union of the sets
$Q_{l,m}$ has full $\mu_f$-measure, we only need to prove the proposition for $\mu_f$-generic points $x\in Q$. 

\vspace{0.16cm}

$\bullet${\sl{ Montel property.--}}
Let $\tilde f : Q \to Q$ be the first return map defined as $\tilde f(x) := f^{r(x)}(x)$ where $r(x)$ is the smallest integer $r\geq 1$ satisfying $f^r(x) \in Q$. 
The induced measure $\tilde \mu(\cdot) := \mu(Q \cap \cdot  )/ \mu(Q)$ is $\tilde f$-invariant and ergodic. 
Let $x$ be a generic point of $Q$. Birkhoff's ergodic theorem, applied to $\tilde f$, yields a sequence $(n_j)_j$ depending on $x$ such that $f^{-n_j}(x)$ is contained in $Q$ and converges to $x$.  To simplify the exposition we avoid the indices of the subsequence and write $f^{-n}(x)$ instead of $f^{-n_j}(x)$. Define $x_n := f^{-n}(x)$, $\eta_n=\eta_{x_n}$, and  $\sigma_n = \sigma_{x_n}$. 

Since $cq \geq 1/l$ on $Q=Q_{l,m}$, one can consider the restriction 
$\eta_n : \disk(1/l) \to \disk(q(x_n)) \subset \disk$. Montel's and Hurwitz's theorems provide  a subsequence (still denoted $\eta_n$)
such that 
\begin{enumerate}
\item[(i)] $\eta_n$ converges towards an injective, holomorphic function   $\eta : \disk(1/l) \to \disk$  such that $\eta (0)=0$ and $\eta '(0)=1$. Moreover,  $\vert \eta(z) \vert \leq l \vert z \vert$ on $\disk(1/l)$ by Schwarz's lemma.
\end{enumerate}

We can also consider the restriction $\sigma_n : \disk(1/l) \to W^{s, loc}(x_n)$. For large enough $n$, its image is contained in the ball $B_x(2)$, and  the image of $\Psi_x \circ \Psi_{x_n}^{-1} \circ \sigma_n$ is a graph above the vertical axis. Moreover $\vert \sigma_n'(0) \vert \geq \beta(x_n) \geq 1/l$. Montel's theorem and the continuity of $\beta$ on $Q = Q_{l,m}$ then yield 
\begin{enumerate}
\item[(ii)] $\sigma_n$ converges towards an injective holomorphic mapping $\sigma : \disk(1/l) \to W^{s,loc}(x)$ such that $\sigma(0)=x$ and $\vert \sigma'(0) \vert  \geq \beta(x) \geq 1/l$.
\end{enumerate} 
Thus,  the function $\xi:=\sigma \circ \eta : \disk(1/l^2) \to W^{s, loc}(x)$ is well defined, injective and satisfies $\sigma \circ \eta (0) = x$; by construction,  the sequence of parametrizations $\xi_n^s:=\sigma_n\circ \eta_n$ converges towards the holomorphic
map $\xi$ on $\disk(1/l^2)$. This map will be our desired $\xi_x$, it is a non-constant holomorphic mapping with values in the
stable manifold $W^s(x)$ (we shall compare it with $\xi_x^s$ in Section \ref{par:densitecste-affine}). Since all $\xi_n^s$ satisfy
the Lipschitz property listed in Section~\ref{par:affine-param}, we get 
\[
2\beta(x)/3 \leq \vert  \xi ' (z)  \vert \leq 2 \textrm{ on } \disk(1/l^2). 
\] 

\vspace{0.16cm}

$\bullet${\sl{ Renormalization.--}}
The identity $(f^n)^* T_f^- = \lambda_f^{-n} T_f^-$ for $n \geq 1$ implies
\[
 (f^n \circ \xi_n^s)^* T_f^- = \lambda_f^{-n} (\xi_n^s)^* T_f^- \ \textrm{ on } \ \disk(1/l^2) .
 \]
Proposition \ref{xi} yields  
\[ (f^n \circ \xi_n^s)^* T_f^- = (\xi_x^s \circ M_n)^* T_f^- \ \textrm{ on } \ \disk(1/l^2),\]
where $M_n$ is defined by $M_n := M_{x_{-1}} \circ \cdots \circ M_{x_{-n}}$. Note that  $M_n(z) = m_n \cdot z$ on $\disk(1/l^2)$ with $\vert m_n \vert \in e^{n\lambda_s} \cdot [e^{-n\epsilon} , e^{n\epsilon}]$. 
Combining these equations, one gets
\begin{equation}\label{poop}
 (\xi_n^s)^* T_f^- =  \lambda_f^{n} \, M_n^* \ (\xi_x^s)^* T_f^-  \ \textrm{ on } \ \disk(1/l^2). 
 \end{equation}
Now, denote by $\varphi_n$ the density of $(\xi_n^s)^*T^-_f$ and by $\varphi_x$ the density for $(\xi_x^s)^*T^-_f$, as in Equation \eqref{densitecourant}. Equation \eqref{poop} gives 
\[
\varphi_n(z) = \lambda_f^n \vert m_n\vert^2 \varphi_x(z).
\]
Since $x_n$ and $x$ are in the set $Q$, the origin is a Lebesgue density point for the densities $\varphi_n$
and $\varphi_x$, and $l^{-1}\leq \varphi_n(0), \; \varphi_x(0)\leq l$. Thus, 
\[
l^{-2}\leq \lambda_f^n   \vert m_n\vert^2\leq l^2.
\]
Taking logarithms and dividing by $n$ leads to
\[
\vert \log \lambda_f + 2\lambda_s   \vert < 2\epsilon
\]
as desired. Moreover, taking a subsequence, one may assume that $\lambda_f^n\vert m_n\vert ^2$ converges to 
a positive real number $\theta\in [l^{-2}, l^2]$. 

Now, we come back to the equation (\ref{poop}) which can be written
\begin{equation}\label{dovdov}
 (\xi_n^s)^* T_f^- = \lambda_f^{n}   \vert m_n \vert^2 \cdot \varphi_x (m_n z) \cdot  \frac{\ii}{2}  dz \wedge d\bar z \ \textrm{ on } \ \disk(1/l^2)  . 
 \end{equation} 
The left hand side converges to $(\xi)^* T_f^-$ in the sense of distributions on $\disk(1/l^2)$, because $T_f^-$ has continuous potentials (see Lemma~\ref{lem:pull-back-cv}). The right hand side converges in the sense of distributions to 
\[
\theta \cdot \varphi_x(0)  \cdot \frac{\ii}{2}  dz \wedge d\bar z,
\]
because $M_n$ converges locally uniformly to the constant mapping $0$ on compact subsets of $\disk(1/l^2)$ and  $0$ is a Lebesgue density point for $\varphi_x$. As a consequence,
\[
 (\xi)^* T_f^- =  \theta \cdot  \varphi_x(0)  \cdot \frac{\ii}{2}  dz \wedge d\bar z \ \textrm{ on } \ \disk(1/l^2)  . 
 \]
Setting $\xi_x=\xi$, $\rho(x) = 1/(l^2)$ and $\alpha(x) = \theta \, \varphi_x(0)$ we get 
 $\xi_x^* T_f^- = \alpha(x) \cdot \frac{\ii}{2}  dz \wedge d\bar z$ on $\disk(\rho(x))$, as desired
 ($\rho(x)$ can be defined as the best constant $l$ for which  $x$ is in $Q_{l,m}$). 
\end{proof}

\subsection{Second step: from $\xi_x$ to $\xi^s_x$}\label{par:densitecste-affine}

We  need to translate Proposition~\ref{densitecste} in terms of the global parame\-trization $\xi_x^s\colon \C\to W^s(x)$. Proposition~\ref{densitecste} asserts that there is a parametrization $\xi_x\colon \disk(\rho(x))\to W^s(x)$ of a
small neighborhood of $x$ in $W^s(x)$ such that $\xi_x^*T^-_f=\alpha(x)(\ii/2)dz\wedge d{\bar{z}}$
on $\disk(\rho(x))$. Both $\xi_x$ and $\xi_x^s$ satisfy the Lipschitz property 
\[
2\beta(x)/3\leq \vert \xi'(z)\vert \leq 2
\]
on their domain of definition, thus $(\xi_x^s)^{-1}\circ \xi_x$ is defined on $\disk(\rho(x))$ and the
modulus of its derivative is bounded from below by $\beta(x)/3$ and from above by $3\beta(x)^{-1}$; moreover, its derivative at the origin is in between $\beta(x)$ and $\beta(x)^{-1}$. 
This function is injective and its domain of definition is $\disk(\rho(x))$. By Koebe one quarter theorem, 
its image contains the disk of radius $\beta(x)\rho(x)/4$, so that the reciprocal function is defined on 
$\disk(\beta(x)\rho(x)/4)$ and has derivative between $ \beta(x)/3$ and $3\beta(x)^{-1}$ on this 
disk. 

In order to pull back $T^-_f$ by $\xi_x^s$, one can first compute its pull back by $\xi_x$,
the result being $\alpha(x) (\ii/2) dz\wedge d\bar z$, and then take its pull-back under 
$(\xi_x)^{-1}\circ \xi_x^s$. This gives
\begin{equation}\label{piip}
(\xi_x^s)^*T^-_f=\alpha(x) \vert h_x(z)\vert \frac{\ii}{2} dz\wedge d\bar z
\end{equation}
on $\disk(\beta(x)\rho(x)/4)$, where $h_x$ is holomorphic and $\vert h_x(z)\vert$
is bounded by 
\begin{equation}\label{eq:bound-on-h}
\beta(x)/3\leq \vert h_x(z)\vert \leq 3\beta(x)^{-1}.
\end{equation}

\subsection{Final step for the proof of Theorem~\ref{thm:pullback-lebesgue}}

Now, fix a Pesin box $\Pes$, and a subset $Q=Q_{l,m}$ as in the proof of Proposition~\ref{densitecste}.
Then, there exists an integer $N$ such that 
\[
1/N \leq   (\beta \rho/4)^2 \leq 1,  \quad 1/N \leq  \beta/3 \leq 1, \quad {\text{and}} \quad1/N \leq  \alpha \leq N 
\]
on the set $Q$. Apply Birkhoff ergodic theorem: for a generic point $x$ of $Q$ there is a sequence of points 
$f^{-n_i}(x)\in Q$ that converges towards $x$. As above, we drop the
index $i$ from $n_i$, and write  
$x_n$ for $f^{-n}(x)$, $\xi_n^s$
for $\xi_{x_n}^s$. As in Equation \eqref{poop}, we obtain
\begin{equation}\label{koko}
 (\xi_x^s) ^*  T_f^- =  \lambda_f^n \, M_n^* \ (\xi_{x_n}^s)^* T_f^-   \textrm{ on } \ \C,
 \end{equation}
 where $M_n$ is a linear map $z \mapsto m_n \cdot z$ with $\vert m_n \vert \in e^{n \lambda_s} \cdot [e^{-n\epsilon},e^{n\epsilon}]$. 
 Equation \eqref{piip} shows that the right hand side of Equation \eqref{koko} has density  
\begin{equation}\label{inter}
 \lambda_f^{n}   \vert M_n' \vert^2 \cdot\alpha(x_n) \cdot \vert h_n ( M_n (z)) \vert^2  
 \end{equation}
on $M_n^{-1} (\disk(1/N))$, where $h_n$ is holomorphic with modulus in $[1/N , N]$. Again, evaluation at $z=0$ gives $\lambda_f^n \vert m_n \vert^2 \cdot \alpha(x_n) \in [1/N^5, N^5]$, so that a subsequence  of $\lambda_f^n \vert m_n \vert^2 \cdot \alpha(x_n)$ converges to some $\theta$ in $[1/N^5, N^5]$. Moreover  Montel's theorem and Equation~\eqref{eq:bound-on-h} imply that the sequence $(h_n)_n$ is equicontinuous on $\disk(1/N)$; thus a subsequence converges locally uniformly to a function $h : \disk(1/N) \to \C$ with modulus in $[1/N , N]$. 

Let $\gamma(x) := \theta \vert h (0) \vert ^2$. Let $K$ be a compact subset of $\C$ and restrict the study to integers  $n \geq 1$ such that $K \subset M_{n}^{-1}({\disk(1/N)})$. Since $(M_n)_n$ converges uniformly to zero on $K$, we get 
\[
(\xi^s_x)^*T^-_f =  \gamma(x) \frac{\ii}{2}dz\wedge d\bar z \ \ \textrm{ on } \ \ K .
\]
The same formula holds on $\C$ by compact exhaustion.
Changing $\xi^s_x$ into $\xi_x^s(az)$ with $\vert a\vert^{-2} := \gamma(x)$, 
we obtain the parametrization promised by Theorem~\ref{thm:pullback-lebesgue}.

\section{Normal families of entire curves}\label{par:NormalStable}

\subsection{Singularities} 

In this section, we need to work on the surface $X_0$, obtained by contracting all periodic curves of $f$. 
This surface may be singular. We refer to \cite{Demailly:1985, Chirka} for all necessary concepts of pluri-potential 
theory on singular analytic spaces. 

Note that, up to now, all arguments were done on the smooth surface $X$, most of them at a generic point for $\mu_f$, but
they could have been done on $X_0$ directly. (we worked on $X$ to simplify the notations).

\begin{rem}\label{rem:inj-proj}
Birkhoff ergodic theorem tells us that the forward orbit of a $\mu_f$-generic point equidistributes with respect to $\mu_f$. 

Let $\xi\colon \C\to X$ be an injective parametrization of a stable (resp. unstable) manifold $W^s(x)$. Assume that
$\xi$ is not contained in a periodic algebraic curve. Let $D$ be an irreducible periodic algebraic curve, and suppose
$f(D)=D$ for simplicity. If $W^s(x)$ intersects $D$, the forward orbit of $x$ converges towards $D$. On the other
hand, $\mu_f(D)=0$ because $\mu_f$ does not charge any proper analytic subset of $X$. Thus, $x$ is not a generic
point with respect to $\mu_f$. 

Thus, in the definition of the set of generic points $\Lambda$ in Oseledets-Pesin theorem (Theorem~\ref{OP}), we can add the hypothesis that
the global stable manifolds $W^s(x)$, $x\in \Lambda$, do not intersect the periodic curves of $f$ (in $X_0$, they do not
go through the singularities). In particular, the injective parametrization $\xi^s_x\colon \C\to X$ remains injective when 
one projects it into $X_0$.
\end{rem}

In what follows, we keep the same notation $\xi^{u/s}_x$ for the unstable and stable manifolds, 
but consider them as entire curves in $X_0$. We also keep the same notation $T^\pm_f$ for
the invariant currents.

\subsection{The family of entire curves $\A^u_f$}\label{norfam}

Let $f$ be an automorphism of a complex projective surface with positive entropy and let $T_f^\pm$ be the invariant currents of $f$ defined in Section~\ref{defprop}. Let $\pi : X \to X_0$ be the morphism that contracts the periodic curves of $f$ (see Proposition~\ref{prop:contraction}). Let $X_0^{reg}$ be the smooth part of $X_0$. Let $T_0^\pm := \pi_* T_f^\pm$; Remark  \ref{pushpush} shows that these currents are well defined and have continuous potentials.
 
\begin{defi}
Let  $\A^u_f$ be the family of entire curves $\xi\colon \C\to X_0$ such that 
\[
\xi^*(T^+_0)=\frac{\ii}{2}dz \wedge d\zbar \quad {\text{and}}\quad \xi^*(T^-_0)=0.
\]
\end{defi}

If $\mu_f$ is absolutely continuous, Theorem \ref{thm:pullback-lebesgue} and Remark~\ref{rem:inj-proj} show that almost every unstable manifold $W^u_f(x)$ can be parametrized by an injective entire curve $\xi^u_x\colon\C\to X_0$ that belongs to $\A^u_f$.
In particular $\A^u_f$ is not empty.

\subsection{Zalcman's theorem and compactness of $\A^u_f$}

Let $M$ be a compact complex space. Fix a hermitian metric on $M$ (see \cite{Demailly:1985}). 

We say that a sequence of entire curves $\xi_n\colon \C\to M$ converges towards  $\xi \colon \C \to M$ if $\xi_n$ converges locally uniformly to $\xi$. 
A family $\E$ of entire curves on $M$ is {\bf{closed}} if the limit of every converging
sequence $(\xi_n)\in \E^\N$ is an element of $\E$. It is {\bf{normal}} if every sequence of elements of $\F$ contains a converging subsequence. It is {\bf{compact}} if it is non empty, normal and closed.

\begin{lem}[Zalcman \cite{Z}] Let $M$ be a compact, complex analytic space. If a sequence of entire curves $\xi_n \colon \C \to M$ is not normal then there exists a sequence of affine automorphisms $z\mapsto a_n z+b_n$ such that 
\begin{itemize}
\item[(1)] $\lim_n a_n = 0$,  
\item[(2)] the sequence $\nu_n \colon \C \to M$ defined by $z \mapsto \xi_n(a_n z +b_n)$ converges towards a Brody curve $\nu\colon \C\to M$.
\end{itemize}
\end{lem}

The following result is a fundamental step in the proof of our main theorem. Its statement does not require any assumption 
on $\mu_f$.

\begin{thm}\label{thm:normal} Let $f$ be an automorphism of a complex projective surface $X$ with positive entropy, 
with $\pi\colon X\to X_0$ the birational morphism that contracts the $f$-periodic curves.  
If non-empty, $\A^u_f$ is a compact family. 
\end{thm}

\begin{proof}
Lemma~\ref{lem:pull-back-cv} implies that $\A^u_f$ is closed. Let us prove that $\A^u_f$ is a normal family. If not, Zalcman's Lemma provides a sequence  $\xi_n\in \A^u_f$ and automorphisms $g_n \colon z \mapsto a_n z+b_n$ such that 
\[
\nu_n\colon z \mapsto \xi_n(a_n z +b_n)
\]
converges towards a Brody curve $\nu\colon \C\to X_0$. This curve satisfies
$$\nu^*(T^+_0) = \lim_{n\to \infty} \nu_n^*(T^+_0) =  \lim_{n\to \infty}¬†g_n^*(\frac{\ii}{2}dz\wedge d\zbar)  =  0  $$
because $\lim_n a_n =0$. Similarly, $\nu^*(T^-_0)=0$. 
This contradicts the second assertion in Corollary~\ref{cor:DS}.
\end{proof}

\subsection{The compact family $\B^u_f$ of unstable manifolds }\label{par:topoB}

We assume that  $\mu_f$ is absolutely continuous, so that $\A^u_f$  is not empty.

\begin{defi}
Let $\B^u_f$ be the smallest compact subset  of $\A^u_f$ that contains all injective parametrizations of unstable manifolds which are in $\A^u_f$. We set 
\[
\B^u_f(X_0)=X_0^{reg}  \bigcap \bigcup_{\xi \in \B^u_f}  \xi(\C)   .
\]
\end{defi}

Note  that $\A^u_f$ and $\B^u_f$ are invariant under translation and rotation: if $\xi$ is an element of one of these
sets, then $z\mapsto \xi(e^{i\theta} z+b)$ is an element of the same set for all $b\in \C$ and $\theta \in \R$.
Using this remark, one verifies that $\B^u_f(X_0)$ is a closed subset of $X_0^{reg}$, because $\B^u_f$ is a 
compact family of entire curves.
The sets $\B^s_f$ and $\B^s_f(X_0)$ are defined in a similar way, with parametrizations of stable manifolds
such that $(\xi^s_x)^*T^-_f= (\frac{\ii}{2})dz\wedge d\overline{z}$.

Let us now derive further properties of  $\B^u_f$ and $\B^u_f(X_0)$.
 
\begin{lem}\label{lem:identity-principle}
Let $\eta_1, \eta_2$ be elements of  $\B^u_f$. Then 
either $\eta_1(\C)$ and $\eta_2(\C)$ are disjoint or $\eta_1(\C) = \eta_2(\C)$. In the later case $\eta_1(z)=\eta_2(e^{i\theta} z+b)$ on $\C$  for some $b\in \C$ and $\theta \in \R$.
\end{lem}

\begin{proof} The first property follows from Hurwitz's lemma (see \cite{BLS1}, Lemma 6.4): 

\begin{lem}[Hurwitz]
Let $C_n$ and $D_n$ be two families of irreducible curves in the unit ball of $\C^2$. Assume that 
$C_n\cap D_n$ is empty for all $n$, that $C_n$ converges to an irreducible curve $C$ uniformly
and that $D_n$ converges to an irreducible curve $D$ uniformly. Then either $C\cap D$ is empty
or $C$ coincides with $D$. 
\end{lem}

To prove the second property, assume that $\eta_1$ and $\eta_2$ have the same image~$W$. Let 
$m$ be a point of $W$ and fix two points $z_1$ and $z_2$ such that $\eta_1(z_1) = \eta_2(z_2)=m$. Assume that $m$ and $z_2$ satisfy $(\eta_2)'(z_2)\neq 0$; one can always find such pairs $(m,z_2)$
because $\eta_2$ is not constant. Then $\eta_2$ determines a local diffeomorphism from 
a  neighborhood   of $z_2$ in $\C$ to a neighborhood  of $m$ in $W$. 
The map $\varphi=\eta_2^{-1}\circ \eta_1$ is defined on a small disk centered at $z_1$, is holomorphic, 
and preserves $(\ii/2) dz\wedge d\bar z$. Thus, $\varphi(z)=e^{i\theta} z+b$ for some $b\in \C$ and $\theta \in \R$. As a consequence, there is a non-empty open 
subset of $\C$ on which $\eta_2(e^{i\theta} z+b)=\eta_1(z)$; this property holds on $\C$ by analytic 
continuation. 
\end{proof}

\begin{lem}\label{lem:topoB}
Let $x$ be an element of $\B^u_f(X_0)$ and $\xi$ be an element of $\B^u_f$ with $x=\xi(0)$. 
Let $\epsilon$ be a positive real number. 
There exists a neighborhood $\V$ of $x$ in $X_0^{reg}$ such that: for every $\eta \in \B^u_f$ with $\eta(0) \in \V$, there exists $\theta \in \R$ satisfying
\[
 \forall z \in \disk_{r_0}, \quad   \dist_{X_0}(\eta(e^{i\theta}z), \xi(z))\leq \epsilon.
\]
\end{lem}

\begin{proof} 
If not, there exists a sequence $\eta_n \in \B^u_f$ such that $\eta_n(0) \in B_x({1 \over n})$ and 
\[ 
\forall \theta \in \R,\,  \exists z_{n,\theta} \in \disk_{r_0}, \, \dist(\eta_n(e^{i\theta}z_{n,\theta}), \xi(z_{n,\theta})) >  \epsilon .
\]
For every angle $\theta \in \R$, choose a limit point $z_\theta \in \overline{ \disk}_{r_0}$ of the sequence $z_{n,\theta}$. By compactness of $\B^u_f$, one can assume that $\eta_n$ converges to some $\eta \in \B^u_f$ with $\eta(0)=x$. By construction, 
\[
\forall \theta \in \R, \;  \dist_{X_0}(\eta (e^{i\theta}z_{\theta}), \xi(z_{\theta})) \geq  \epsilon .
 \]
In particular $\eta$ is not equal to $\xi$ up to a rotation, contradicting Lemma~\ref{lem:identity-principle}.
\end{proof}

\begin{pro}\label{pro:stable/B}
Assume that the measure $\mu_f$ is absolutely continuous with respect to Lebesgue measure. 
Let $\xi$ be an element of $\B^u_f$. Then 
\begin{enumerate}
\item $\overline{\xi(\C)}\cap X_0^{reg}$ is contained in $\B^u_f(X_0)\cap \B^s_f(X_0)$;
\item $\forall z\in \C$, $\xi(z)$ is contained in the image of an entire curve $\nu\in \B^s_f$;
\item the set of points $z\in \C$ such that $\xi$ intersects a stable manifold $\xi^s_y$ of $f$ 
transversely at $\xi(z)$ is dense in $\C$.
\end{enumerate}
The same result holds if one permutes the roles of stable and unstable parame\-trizations.
\end{pro}

\begin{proof} (recall that, for simplicity, the current $\pi_*(T^\pm_f)$ is still denoted $T^\pm_f$)

Assertion (2) is weaker than assertion (1). Thus, we only prove (1) and (3). 
First, by definition, $\xi(\C)\cap X_0^{reg}$ is contained in $\B^u_f(X_0)$. Since $\B^u_f(X_0)$ is a closed subset of $X_0^{reg}$, it contains ${\overline{\xi(\C)}}\cap X_0^{reg}$; similarly, $\B^s_f(X_0)$ contains $\overline{\eta(\C)}\cap X_0^{reg}$ for all curves $\eta\in \B^s_f$.
Since $\xi^* T^+_f$ coincides with Lebesgue measure, it has full support. But $T^+_f$ is an Ahlfors-Nevanlinna current for every stable manifold $\xi^s_x\colon \C\to X_0$ (see \S~\ref{TequalA}). Taking $\xi^s_y$ in $\B^s_f$, we obtain 
\[
\xi(\C)\subset {\overline{\xi^s_y(\C)}} \subset \B^s_f(X_0),
\]
and the first assertion follows from these inclusions. 

To prove assertion (3), we use the same curve $\xi^s_y\in \B^s_f(X_0)$. Recall that $T^+_f$ is laminar and has continuous potentials. By \cite[section 8]{BLS1}, this implies that the slice of $T^+_f$ by $\xi$ is geometric; in particular, $\xi^s_y(\C)$ intersects $\xi(\C)$ on a dense subset. By Lemma~\ref{lem:BLS6.4}, each of these intersections can be approximated by transverse ones. This  completes the proof of the proposition.
\end{proof}

\subsection{Local laminations}

Every $\xi \in \B^u_f$ is a uniform limit of (generic) unstable manifolds $\xi_n^u \in \A^u_f$; in particular, $\xi^u_n$ takes values in $X_0^{reg}$ and $(\xi^u_n)'(z) \neq 0$ for every $z \in \C$. But it may happen that $\xi'(z)=0$ for some $z\in \C$. For instance, 
in (singular) Kummer examples, the velocity vanishes for the stable and unstable manifolds when they pass through
the singularities of the surface.

\begin{defi} The {\bf{velocity}} $v(x)$ at a point $x \in \B^u_f(X_0)$ is defined by 
\[v(x) := \parallel \xi'(0) \parallel, 
\] 
where $\xi$ is any element of $\B^u_f$ such that $\xi(0)=x$. 
The set $\B^u_f(X_0)$ is partitioned into the set 
\[
\B^{u,+}_f(X_0) := \left\{ x\in \B^u_f(X_0) \, , \, v(x) > 0 \right\}
\]
and 
\[\B^{u,0}_f(X_0) := \left\{ x\in \B^u_f(X_0) \, , \, v(x) = 0 \right\}.\] 
\end{defi}

The fact that the velocity is well defined, i.e. does not depend on the choice for $\xi$, follows 
from Lemma \ref{lem:identity-principle}.

\begin{pro}\label{pro:local-lami}
Assume that $\mu_f$ is absolutely continuous. Let $x$ be a point in $\B^{u,+}_f(X_0)$ and $\xi$ be an element of $\B^u_f$ such that $\xi(0) =x$. 
There are neighborhoods $\U\subset \U'$ of $x$ in $X_0^{reg}$ such that
\begin{enumerate}
\item $\U'$ is isomorphic to a bidisk $\disk \times \disk$.
\item The connected component of $\xi(\C)\cap \U'$ that contains
$x$ is a horizontal graph in $\U$.
\item There is a lamination $\L^u$ of the whole open set $\U$ by horizontal graphs, 
each of which is contained in the image of a curve $\eta\in \B^u_f$. In particular $x$ is an interior point of $\B^u_f(X_0)$ in the complex surface $X_0^{reg}$.
\item If $\eta$ is an element of  $\B^u_f$, $\Delta$ is an open subset of $\C$
and $\eta(\Delta)$ is contained in $\U$ then $\eta(\Delta)$ is contained in a leaf
of this lamination.
\item There is a transversal to the lamination $\L^u$ which is a piece of a stable manifold of $f$.
\item The support of $\mu_f$ in $\U$ coincides with $\U$.
\end{enumerate}
\end{pro}

\begin{figure}[h]
\centering\epsfig{figure=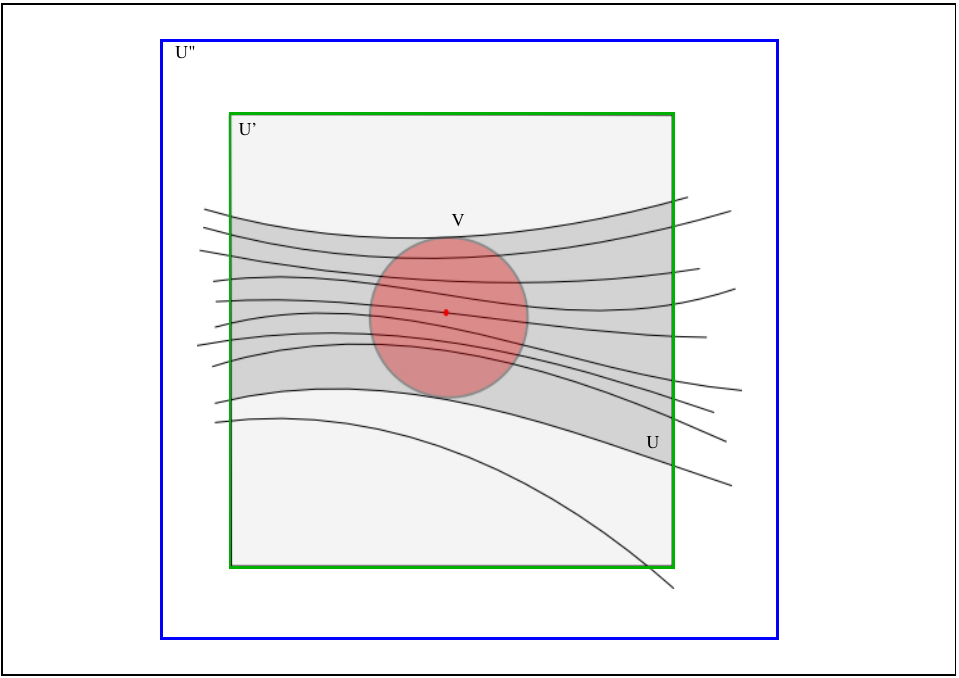}
\caption{{\small{The point $x$ is the red point right in the middle. The open set $\V$ is the red shaded ball, the open set $U$ is grey; $U'$ is light grey with a green contour, and $U"$ is white with a blue contour. The horizontal curves represent pieces of unstable manifolds (one of them is not a graph above $\disk$).}}}
\end{figure}

\begin{proof}
By definition, $\xi'(0)\neq 0$.
Hence, there exist $r_0>0$ and neighborhoods $\U'\subset \U''$ of $x$ in $X_0^{reg}$ such that
\begin{itemize}
\item $\xi(\disk_{r_0})$ is a smooth curve,
\item  the pair of open sets $(\U', \U'')$ is isomorphic to a pair of bidisks $(\disk \times \disk,\disk_R\times \disk_R)$, with $R>1$,
\item  $\xi(\partial \disk_{r_0}) \subset X_0^{reg} \setminus \U' $ and $\xi(\disk_{r_0})$ is contained in $\U''$,
\item  $\xi(\disk_{r_0})\cap \U'$ is a horizontal graph in the bidisk $\U'$.
\end{itemize}
Changing $(\U', \U'')$ if necessary there exists $r< r_0$ such that $\xi(\disk_{r_0})\cap \U'=\xi(\disk_r)$. 

\begin{lem}\label{graf}
Let $r_0 > 0$ and $\U'$ as above. 
There exists a neighborhood $\V$ of $x$ in $X_0^{reg}$ such that for every $\eta \in \B^u_f$ with $\eta(0) \in\V$, the curve $\eta(\disk_{r_0})\cap \U'$ is a horizontal graph in $\U'$.
\end{lem}

\begin{proof}
Apply Lemma~\ref{lem:topoB}: given $\epsilon>0$, there is a neighborhood $\V$ of $\xi(0)$ such that every curve 
$\eta\in \B^u_f$ with $\eta(0)\in V$ is $\epsilon$-close to $\xi$ on $\disk_{r_0}$ after composition
with a rotation $z\mapsto e^{i\theta}z$. All we need to prove, is that $\eta$ is also a horizontal graph if
$\epsilon$ is small enough. 

Let $\pi\colon \U''\simeq \disk_R\times \disk_R\to \disk_R$ denote the first projection. The holomorphic mappings $\pi\circ \xi$ 
and $\pi\circ \eta$ are $\epsilon$-close on $\disk_{r_0}$. On the smaller disk $\disk_r$, the map $\pi\circ\xi$ is one-to-one, with 
image equal to the unit disk $\disk$. If $\epsilon$ is small enough, $\vert \pi\circ \eta-\pi\circ \xi\vert < \vert \pi\circ\xi\vert$
on the boundary of $\disk_{r_0}$. Thus, Rouch\'e theorem implies that $\pi\circ \eta\colon\disk_r\to \disk$ is also one-to-one, and the result follows. 
\end{proof}

Fix $\V$ as in the previous lemma. 
The curves $\eta(\disk_{r_0})\cap \U'$ with $\eta \in \B^u_f$ and $\eta(0) \in \V$ form a family 
of horizontal graphs in $\U'$. Let $\U$ be the union of these graphs. By Lemma \ref{lem:identity-principle} it is laminated by disjoint horizontal graphs; we denote this lamination by $\L^u$. It remains to prove that $\U$ is an open 
neighborhood of $\xi(\disk_{r})$ in $\U'$. 

Apply Proposition~\ref{pro:stable/B}, Assertion~(3), to the curve $\xi$. One can find an element $\xi^s_y$ of $\B^s_f$ such that 
$\xi^s_y(0) \in \xi(\disk_r)\cap\V$, with transverse intersection. Let $\Delta$ be a disk centered at the origin for which
$\xi^s_y(\Delta)$ is contained in $\V$. Apply Proposition~\ref{pro:stable/B}, Assertion (1), but to the stable manifold $\xi^s_y$: every point of $\xi^s_y(\Delta)$ is contained in the image of a curve $\eta\in \B^u_f$. Thus, the set 
$\U$ contains $\xi^s_y(\Delta)$. But $\U$ is laminated, and the $\lambda$-lemma implies that the holonomy of a lamination 
are (quasi-conformal) homeomorphisms. Thus, $\U$ contains a neighborhood of $\xi(\disk_r)$.

For property (5), shrink $\Delta$ and $\U$ to assure that $\xi^s_y(\Delta)$ is  transverse to~$\L^u$. 

The previous argument shows that the support of the restriction $T^-_{f\vert \U}$   coincides with $\U$ because its slice
$\xi^s_y(T^-_f)_{\vert \Delta}$ is the Lebesgue measure (see the proof of Proposition~\ref{pro:stable/B}).
To prove property (6), fix a small ball $B\subset \U$. Since $T^-_f$ charges $B$, 
there is an unstable manifold that enters $B$. Take a point $x'\in B$ on this unstable manifold which intersects a stable manifold transversally, and apply properties (1) to (5) for the stable manifolds: we get a stable lamination $\L^s$ of a neighborhood ${\mathcal{W}}$ of $x'$ which is transverse to $L^u$. In ${\mathcal{W}}$, the product of $T^-_f$ and $T^u_f$ is strictly positive; hence $\mu_f(B)>0$.
\end{proof}

\section{Holomorphic foliation and Hartogs extension}\label{par:CONCLUSION}

We  conclude the proof of the main theorem. The first step is to promote the local laminations obtained in Proposition \ref{pro:local-lami} into local holomorphic foliations. Then we use Hartogs phenomenon to extend the foliation into a global, singular foliation of $X$. The theorem eventually follows from a classification of automorphisms preserving holomorphic foliations. 

\subsection{From laminations to holomorphic foliations}\label{par:lam-to-holo}

Proposition \ref{pro:local-lami} asserts that $\B^u_f$ determines near every $x \in \B^{u,+}_f(X_0)$ a local lamination denoted by $\L^u$. In this section we promote these local laminations into local holomorphic foliations. 

\begin{pro}\label{pro:from-lami-to-foli}
The local laminations $\L^u$ are holomorphic. 
\end{pro}

\begin{proof} The following argument is due to Ghys (see \cite{Ghys:Inventiones, Cantat:Acta}). 

Let $\U$ and $ \U'$ be connected open subsets of $X_0^{reg}$ such that $\U'$ is biholomorphic to a bidisk $\disk\times \disk$, 
$\U$ is contained in $\U'$,  and the lamination $\L^u$ of $\U$  is made of disjoint horizontal graphs (as in Fig. 1). Denote by $(x_1,x_2)$ the
coordinates in $\U'\simeq \disk\times \disk$ and by $\pi_1$ and $\pi_2$ the two projections ($\pi_i(x_1,x_2)=x_i$, $i=1,2$).
 Let $p=(p_1,p_2)$ be a point of $\U$ and $\Delta_p\subset \U'$ be the vertical curve $\{x_1=p_1\}$. Denote 
 by $\gamma$ a continuous path in the leaf through $x$ that starts at $x$ and ends at another point 
 $q=(q_1,q_2)$. The holonomy map
from the transversal $\Delta_p$ to the transversal $\Delta_q :=\{x_1=q_1\} $ is a homeomorphism $h_\gamma$ from $\Delta_{p,\U} :=\Delta_p\cap \U$ to 
its image $\Delta_{q,\U} :=\Delta_q\cap \U$. According to the $\Lambda$-lemma,  this homeomorphism $h_{p,q}$ is 
$K(p,q)$-quasi-conformal, with a constant $K(p,q)$ satisfying
\[
0\leq K(p,q)-1\leq \dist(p,q).
\]
Recall that a $1$-quasi-conformal map is conformal, hence holomorphic (see \S~\ref{quasico}). 

Instead of looking at vertical disks, we first choose pairs of disks 
$\Delta$ and $\Delta'$ which are contained in stable manifolds of $f$ and are transverse to the lamination. 
Let $\gamma$ be a path in a leaf $\L^u(m)$ that joins the intersection points $x :=\Delta\cap \L^u(m)$ and
$x' :=\L^u(m)\cap \Delta'$, let $h_\gamma$ be the (germ of) holonomy from $\Delta$ to $\Delta'$. Applying 
$f^{-n}$ and using Remark~\ref{rem:affpara2} there exists $\alpha_n \in \C^*$ such that 
\[
f^{-n}\circ \xi^u_x(z)=\xi^u_{f^{-n}(x)}(\alpha_n z) .
\] 
Proposition \ref{affpara} precisely asserts that $\lim_n \alpha_n   =0$ for almost every point $x$. 

Using Poincar\'e recurrence theorem we can assume that $f^{-n}(x) \in \U$. This and the compactness of $\B^u_f$ imply that the path $\gamma$ is shrunk uniformly under the action of $f^{-n}$; simultaneously, 
the disks $\Delta$ and $\Delta'$ are mapped to large disks $f^{-n}(\Delta)$ and $f^{-n}(\Delta')$: the connected components
of $f^{-n}(\Delta)\cap \U$ and $f^{-n}(\Delta')\cap \U$ containing $f^{-n}(x)$ and $f^{-n}(x')$ are (smaller) vertical  disks in $\U$, and their
relative distance goes to $0$ with $n$. Thus, the holonomy $h_{-n}$ between these disks is $K_n$-quasi-conformal, with $\lim_n K_n = 1$. 
But $f^n$ conjugates $h_{-n}$ with $h_\gamma$ on a small neighborhood of $x$ in $\Delta$. Hence, for almost every point $x \in \Delta$ (with respect to the conditional measure of $\mu_f$ and therefore also with respect to Lebesgue measure), and for every $\epsilon>0$, the holonomy
$h_\gamma$ is $(1+\epsilon)$-quasi-conformal on a small neighborhood $\Delta_\epsilon\subset \Delta$ of $x$. This implies that $h_\gamma$ is $(1+\epsilon)$-quasi-conformal for all $\epsilon>0$; hence, $h_\gamma$ is indeed conformal. 

We have proved that the holonomy between two transversal disks which are contained in stable manifolds is holomorphic. By Proposition~\ref{pro:stable/B} these transversal disks form a dense subset of transversals. Consequently, the holonomy between all pairs of vertical disks $\Delta_x$ and $\Delta_y$ is holomorphic. This implies that
the lamination $\L^u$ is holomorphic. \end{proof}

\subsection{A holomorphic foliation $\F^u$ on $\B^{u,+}_f(X_0)$}\label{gluglu}

\begin{lem}
The local holomorphic foliations $\L^u$ defined near every point of $\B^{u,+}_f(X_0)$ can be glued together to provide a holomorphic foliation $\F^u$ of the open set $\B^{u,+}_f(X_0)$. 
\end{lem}

\begin{proof}
Propositions~\ref{pro:local-lami}  and \ref{pro:from-lami-to-foli} show that if $x$ is in $\B^{u,+}_f(X_0)$, then there exists a neighborhood $\U$ of $x$ and a holomorphic foliation $\L^u$ of $\U$ such that all entire curves $\eta\in \B^u_f$ are tangent to $\L^u$. More precisely, if $\eta$ is an element of $\B^u_f$ and $\eta(\Delta)$ is contained in $\U$, then 
$\eta(\Delta)$ is contained in a leaf of $\L^u$; in other words, if $\omega^u$ is a holomorphic $1$-form on $\U$ which defines the foliation $\L^u$, then $\eta^*\omega^u=0$ on $\Delta$. The holomorphic foliation $\F^u$ can therefore be defined as the unique
holomorphic foliation of $\B^{u,+}_f(X_0)$ such that the generic unstable manifolds of $f$ are leaves of $\F^u$.
\end{proof}

\subsection{Hartogs extension of the holomorphic foliation $\F^u$}\label{VZ}
The holomorphic foliation $\F^u$ is defined on the open set $\B^{u,+}_f(X_0)$. Now, we extend it to $X_0^{reg}$ (i.e. to $\B^{u,0}_f(X_0)$).

\begin{pro}\label{pro:localZ}
Let $x$ be an element of $\B^{u,0}_f(X_0)$. There exists a neighborhood $\V$ of $x\in X_0^{reg}$ such that
the holomorphic foliation $\F^u$ extends as a (singular) holomorphic foliation of $\V$. 
\end{pro}

To prove this proposition, fix a curve $\xi \in \B^u_f$ with $\xi(0)=x$ (and $\xi'(0)=0$). 

\begin{rem} The curve $\xi(\disk)$ may have a singularity at $x$,  but
Proposition 15 of \cite{Lyubich-Peters} rules out this eventuality (see also \cite{Berger-Dujardin:preprint}). On the other hand,  this
argument does not exclude the possibility of a vanishing velocity.  (thanks to Misha Lyubich 
for pointing out this reference)
\end{rem}

Even if $\xi(\disk)$ is singular, there exists an open neighborhood $\V$ 
of $x$ such that 
\begin{itemize}
\item[(i)] $\V\subset \C^2$ up to a local choice of co-ordinates, 
\item[(ii)] the connected component of $\xi^{-1}(\V)$ containing $0$ is a disk $\disk_r$, 
\item[(iii)] $\xi'(0)=0$ but $\xi'$ does not vanish on $\disk_r\setminus\{0\}$. 
\end{itemize}
Fix a radius $s$ with $0< s< r/3$. Since $\xi'$ does not vanish on $\partial\disk_s$, there is an open neighborhood 
$\U\subset \V$ of $\xi(\partial \disk_s)$ and a holomorphic foliation $\L^u$ on $\U$ such that all entire curves
$\eta\in \B^u_f$ are tangent to $\L^u$ in $\U$. 

Fix $\epsilon>0$. Let $\theta \in \R$ and define $x_\theta := \xi(se^{i\theta})$. 
By Lemma~\ref{lem:topoB} there exists a neighborhood $\U(\theta,\epsilon)\subset \U$ of $x_\theta$
such that, if $\eta\in \B^u_f$ and $\eta(se^{i\theta})\in \U(\theta, \epsilon)$, then there exists a complex number $a$ with $\vert a \vert =1$ such that  
$\eta(z)$ and $\xi(az)$ are $\epsilon$-close on $\disk_r$. In particular, the curve $\eta(\disk_{2s})$ is contained in an $\epsilon$-neighborhood of 
$\xi(\disk_r)$. 

Thus, there is a neighborhood $\U_0\subset \U$ of $ \xi(\partial \disk_s)$ satisfying
\begin{itemize}
\item $\forall \, m\in \U_0$, the leaf of $\L^u$ through $m$ is contained in a disk $\eta(\disk_{2s})$ where
$\eta$ is an element of $ \B^u_f$ which is $\epsilon$-close to $\xi$ on $\disk_r$;
\item if $(x_n) \in \U_0^\N$ converges towards a point of $\xi(\disk_s)$, 
one can choose such disks $\eta_n(\disk_{2s})$ with $\eta_n\in \B^u_f$ converging towards $\xi$ 
uniformly on $\disk_r$.
\end{itemize}

Let $(x_n)_{n\geq 1}$ be such a sequence, with the additional property that $x_n$ is contained in an unstable
manifold of $f$ for all $n$. Then, $\eta_n$ is a parametrization of an unstable manifold, and therefore
$\eta_n'$ does not vanish. Hence, there is an open neighborhood $\U_n$ of $\eta_n(\disk_s)$ such that $\L^u$ extends as a holomorphic foliation of $\U_0\cup \U_n$. As explained in Section~\ref{gluglu}, the extensions of $\L^u$ to the open sets $\U_n$ are compatible, and determine a foliation of $\U_\infty := \cup_{n \geq 0} \U_n$. 

Note that $\U_\infty$ is a subset of $\V \subset \C^2$. The slopes of the leaves of
$\L^u$ determine a holomorphic function ${\sf{s}}^u\colon \U_\infty \to \P^1$ (with $\P^1$ the projective
line of all possible ``slopes''). Such a function extends as a meromorphic function $\widehat{{\sf{s}}^u}$ on 
the envelop of holomorphy $\widehat{\U_\infty}$  of $\U_\infty$ (see \cite{Ivashkovich}). The function  $\widehat{ {\sf{s}}^u}$ determines a (singular) holomorphic foliation on $\widehat{\U_\infty}$ which extends $\L^u$. 

It remains to show that $\widehat{\U_\infty}$ contains a neighborhood of $\xi(\disk_s)$, hence a neighborhood of $x$. For this purpose, we apply the following theorem to $D=\U_\infty$, $S_n=\eta_n(\disk_s)$ and $S=\xi(\disk_s)$, and we remark that the boundaries $\eta(\partial \disk_s)$ are contained in a compact neighborhood $K$ of $\xi(\partial \disk_s)$ with $K\subset \U_0$.

\begin{thm}[Behnke-Sommer, \cite{Chabat:vol2}, chapter 13] Let $D$ be a bounded domain of $\C^m$, $m\geq 2$. Let $S_n$ 
be a sequence of complex analytic curves which are properly contained in $D$. Assume that $S_n$ converges
towards a curve $S\subset \C^m$ and that the boundaries $\partial S_n$ converge to a curve $\Gamma\Subset D$. 
Then every holomorphic function $h\in {\mathcal{O}}(D)$ extends to a neighborhood of $S$. 
\end{thm}

Since, on a surface, the singularities of a holomorphic foliation are isolated (they correspond to the indeterminacy points
of the ``slope function $\hat{s^u}$''), there is an open neighborhood of 
$x$ in which $\L^u$ has at most one singular leaf, namely the leaf $\xi(\disk_r)$.  

\begin{cor}
The set $\B^u_f(X_0)$ coincides with $X_0^{reg}$ and the foliation $\F^u$ extends to a (singular) holomorphic foliation of 
$X_0^{reg}$, and this foliation is $f$-invariant. Its lift to $X$ by the birational morphism $\pi\colon X\to X_0$ 
determines a (singular) $f$-invariant foliation of $X$.
\end{cor}

\begin{proof}
The set $\B^u_f(X_0)$ is closed. Propositions \ref{pro:local-lami} and \ref{pro:localZ} show that $\B^u_f(X_0)$  is also open. 
Since $X_0^{reg}$ is isomorphic to the complement of finitely many curves in $X$, it is connected. 
Thus, $\B^u_f(X_0)$ is equal to $X_0^{reg}$, and $X_0^{reg}$ supports a (singular) holomorphic foliation $\F^u$, such that
every unstable manifold of $f$ is a leaf of $\F^u$. This implies that $f$ preserves $\F^u$. 

It remains to show that $\F^u$ lifts to a holomorphic foliation of $X$. In the complement $X'$ of the periodic curves
of $f$, the foliation $\pi^* \F^u$ is a well-defined holomorphic foliation. Given an open covering $\U_i$ of $X'$, 
the foliation $\pi^*\F^u$ is defined by local holomorphic vector fields $v_i$ with isolated zeroes, such that 
on $\U_i\cap \U_j$
\[
v_i=g_{i,j}v_j
\]
for some non-vanishing holomorphic function $g_{i,j}\in {\mathcal{O}}^*(\U_i\cap \U_j)$. This cocycle determines a
line bundle on $X'$: the cotangent bundle $T^*_{\pi^*\F^u}$ of $\pi^*\F^u$. Since the periodic curves can be contracted, 
they are contained in arbitrarily small pseudo-convex open subsets of $X$, and the line bundle $T^*_{\pi^*\F^u}$
can be extended to a line bundle ${\mathcal{T}}$ on $X$. The relations $v_i=g_{i,j}v_j$ can be thought of as defining
relations of a holomorphic section $v$ of ${\mathcal{T}}\otimes TX$ on the open subset $X'$. Again, this section 
extends to a global holomorphic section on $X$. This section corresponds to local vector fields  extending $\pi^*\F^u$ 
to a global, $f$-invariant, (singular) holomorphic foliation on $X$. 
\end{proof}

\subsection{Proof of the main theorem}
 
To complete the proof of the main theorem, we refer to the articles  \cite{Cantat:Acta} and \cite[Th\'eor\`eme 3.1]{Cantat-Favre}, in which the following theorem is proved:
\begin{thm}
 Let $X$ be a compact K\"ahler surface, with a (singular) holomorphic foliation $\F$ and an automorphism $f\colon X\to X$
 of positive entropy. If $f$ preserves $\F$, then $(X,f)$ is a Kummer example, and the stable (or unstable) manifolds of $f$ are leaves of $\F$.
 \end{thm}
 
The results of \cite{Cantat:Acta, Cantat-Favre} are more general: they classify triples $(X,\F,f)$ where $X$ is a smooth surface, $\F$ is a (singular) holomorphic foliation of $X$, and $f$ is a birational transformation of $X$ of infinite order preserving  $\F$. 
 
\section{Consequences and applications}\label{par:Consequences}

\subsection{Equivalent dynamical characterizations} 
As before, consider an automorphism $f$ of a complex projective surface $X$, with positive entropy $\log \lambda_f$.

\subsubsection{Ruelle's inequalities and absolute continuity}

The first part of the following result is due to Ruelle. The second part is proved by Ledrappier in \cite[Corollaire 5.6]{L}, in a more general setting. Here the local product structure of $\mu_f$ leads to a somewhat simplified proof. 

\begin{pro}[Ruelle, Ledrappier]\label{ledproduct}
Let $X$ be a complex projective surface and $f$ be an automorphism of $X$ with positive entropy $\log \lambda_f$. Then the Lyapunov exponents of $f$ with respect to $\mu_f$ satisfy 
\[
\lambda_s  \leq -{1 \over 2} \log \lambda_f \ \ {\text{ and }} \quad  \lambda_u  \geq  {1 \over 2} \log \lambda_f .
\]
If equality holds simultaneously in these two inequalities, then $\mu_f$ is absolutely continuous with respect to Lebesgue measure.
\end{pro}

\begin{proof}
Ruelle's inequality states that $\log (\lambda_f) \leq 2 \, \lambda_u$ and $\log (\lambda_f) \leq - 2 \, \lambda_s$ (see \cite{R} and \cite{Katok-Hasselblatt}, p. 669).  Assume that equality holds simultaneously in these two inequalities.  Fix a Pesin box $\Pes$, as in Section~\ref{pesinboxes}. 
According to 
 \cite[Th\'eor\`eme 4.8]{L}, the conditional measures of $\mu_f$ with respect to the stable and unstable manifolds are absolutely continuous with respect to Lebesgue measure. In other words, both $\nu^+$ and $\nu^-$ are absolutely continuous with 
 respect to Lebesgue measure, because these measures coincide with the conditional measures of $\mu_f$ with respect to the stable and unstable laminations of $\Pes$.  
In the Pesin box,  $\mu_{f }$ corresponds to the product measure $\nu^+\otimes \nu^-$ via the homeomorphism $h$ of \S~\ref{par:pesinboxesdef}. Since the holonomy of the stable and unstable laminations are quasi-conformal, they are absolutely continuous with respect to Lebesgue measure (see \S~\ref{par:Pesin-AbsoC}). Hence, $\mu_f$ is absolutely continuous with respect to the Lebesgue measure in $\Pes$, and therefore in $X$. \end{proof}

\subsubsection{Proof of Corollary \ref{cor1}}

Our main theorem provides $(1) \Rightarrow (4)$, while the reverse implication is obvious. We proved $(2) \Rightarrow (1)$ in Proposition \ref{ledproduct} and $(1) \Rightarrow (2)$ in Proposition \ref{densitecste} (recall that this equivalence is also a consequence of \cite[Corollaire 5.6]{L}).  To prove $(2) \Leftrightarrow (3)$ we use Ruelle's inequality \cite{R}, which provides $\lambda_s \leq -{1 \over 2} \log \lambda_f$ and $\lambda_u \geq {1 \over 2} \log \lambda_f$, and Young's theorem \cite{Y2} which ensures that the generic limits in property $(3)$ are equal to $(1/\lambda_u - 1/\lambda_s) \log \lambda_f$. (note that Young's theorem is proved for ${\mathcal{C}}^\infty$-diffeomorphisms of compact surfaces, but her proof applies also to our context)

\subsection{K3 and Enriques surfaces}

\subsubsection{The Classification Theorem}

The Classification Theorem stated in the introduction is proved in \cite{Cantat-Favre, Cantat-Favre:II}. Let us add two remarks. 
Assertion (1) of the Classification Theorem rules out the case of Enriques surfaces. In particular, if $f$ is an automorphism of
an Enriques surface with positive entropy, then $\mu_f$ is singular with respect to Lebesgue measure. 
Assertion (3) is sharp, meaning that there are rational Kummer examples with $\lambda_f$ in $\Q(\zeta_l)$ for all possible
orders $l=3$, $4$, and $5$. 
For instance, an example is given in \cite{Cantat-Favre:II}, of an automorphism $g$ of an abelian surface $A$ such that 
$\lambda_g=\vert1+\zeta_5\vert^2$, where $\zeta_5$ is the primitive fifth root of unity $\exp(2\ii \pi/5)$. The linear transformation 
$(x,y)\mapsto (\zeta_5 x,\zeta_5 y)$ induces an automorphism of $A$, and the quotient is a rational surface: 
this gives examples of Kummer automorphisms on rational surfaces for which $\lambda_f$ has degree $4$.

\subsubsection{Proof of Corollary \ref{cor:K3Pic2}}
This corollary  is  a direct consequence of the second assertion of the Classification Theorem and 
the following classical result.

\begin{lem}\label{lem:Torelli-Pic2}
Let $X$ be a complex projective K3 surface, with Picard number equal to $2$. The group
of automorphisms of $X$ is infinite if and only if the intersection form does not represent $0$
and $-2$ in $\NS(X)$. If it is infinite, then it is virtually cyclic, and all elements of $\Aut(X)$
of infinite order have positive topological entropy. 
\end{lem}

\begin{proof}[Sketch of proof]
Let $\Isom(\NS(X))$ be the group of isometries of the lattice $\NS(X)$ with respect to the intersection form $\langle\cdot \vert \cdot\rangle$.

{\bf{Step 1.--}} By Hodge index theorem the intersection form has signature $(1,1)$ on $\NS_\R(X)$. Assume that this 
form represents $0$; this means that the two isotropic lines of $\langle\cdot \vert \cdot\rangle$ are defined
over $\Z$: they contain primitive elements $v_1$ and $v_2$ in $\NS(X)$. Since the isotropic
cone is $\Isom(\NS(X))$-invariant and the automorphisms of $\Z$ coincides with $\pm \Id$, a subgroup 
of index at most $4$ in $\Isom(\NS(X))$ preserves the two isotropic lines pointwise. Thus, $\Isom(\NS(X))$
has at most four elements. On the other hand, every element $f$ of $\Aut(X)$ determines an element $f^*$ 
in $\Isom(X)$ and the morphism $f\mapsto f^*$ has finite kernel (because the group of automorphisms of 
a K3 surface is discrete). Thus, if the intersection form represents $0$, the group $\Aut(X)$ is finite. 

{\bf{Step 2.--}} Now assume that $\langle\cdot \vert \cdot\rangle$ does not represent $0$. Consider the subgroup $\Isom(\NS(X))^+$ 
of $\Isom(\NS(X))$ that fixes the connected component $H$ of $\{u\in \NS_\R(X);\;  \langle u \vert u \rangle >0\}$ containing
ample classes. This group is infinite, and it is either cyclic, or dihedral (this is equivalent to the resolution of Pell-Fermat 
equations); more precisely, a subgroup of $\Isom(\NS(X))^+$ of index at most $2$ is generated by a hyperbolic isometry $\psi$, which dilates one of the isotropic lines by a factor $\lambda_\psi>1$ and contracts the other one by $1/\lambda_\psi$.

If the ample cone of $X$ coincides with $H$, Torelli theorem shows that the image of $\Aut(X)$ in 
$\Isom(\NS(X))$ is a finite index subgroup of $\Isom(\NS(X))^+$ (see \cite{BPVDVH}). Since the kernel of $f\mapsto f^*$ is finite, $\Aut(X)$
is virtually cyclic and, if $f$ is an automorphism of $X$ of infinite order, $f^*$ coincides with an iterate of $\psi^l$, $l\neq 0$. 
Thus, the topological entropy of $f$ is equal to $\vert l \vert \log(\lambda_\psi)$ and is  positive. 

{\bf{Step 3.--}} The ample cone of $X$ is the subset of classes $a$ in the cone $H$ such that $\langle a \vert [E]\rangle >0$ 
for all irreducible curves $E\subset X$ with negative self-intersection. But, on a K3 surface, such a curve is
a smooth rational curve with self-intersection $-2$. Thus, $H$ coincides with the ample cone if and only if
$X$ does not contain any $-2$-curve. On the other hand, Riemann-Roch formula implies that $X$ contains 
such a $-2$-curve if and only if the intersection form represents $-2$ on $\NS(X)$. To sum up, if $\langle\cdot \vert \cdot\rangle$
represents $-2$, the ample cone is a strict sub-cone of $H$. In that case, the group of isometries of $\NS(X)$ preserving
both $H$ and the ample cone is finite, so that $\Aut(X)$ is finite too.
\end{proof}

\subsubsection{Proof of Corollary \ref{cocor}} 

Let $Z$ be an Enriques surface.
The N\'eron-Severi group $\NS(Z)$ and its
intersection form  $q_Z$ form a non-degenerate lattice of dimension $10$ and signature $(1,9)$; this lattice is isomorphic
to ${\mathbb{U}}\oplus (-{\mathbb{E}}_8)$. For a generic Enriques surface, the action of $\Aut(Z)$ on 
$\NS(Z)$ is an embedding that realizes $\Aut(Z)$ as a finite index subgroup in ${\sf{O}}_{q_Z}(\Z)$: the image
coincides with the subgroup of matrices $B$ that preserve each connected component of the set $\{u\in \NS(Z)\otimes \R\vert \; q_Z(u)>0\}$ and are equal to the identity modulo $2$ (see \cite{Barth-Peters:1983}). Thus, $\Aut(Z)$ contains non-abelian free
groups and contains many elements $f$ with $\lambda_f>1$.
To complete the proof it suffices to apply Lemma~\ref{lem:CFKummer} (or, to notice that Enriques surfaces depend on $10$ parameters and tori depend on $4$ parameters).

\begin{rem}
There are examples of Enriques surfaces with finite automorphism group. There are examples for which 
$\Aut(Z)$ is infinite but no automorphism has positive entropy (see \cite{Barth-Peters:1983},  Section~4; in that example, $\Aut(Z)$
is finite by cyclic). If $\Aut(Z)$ is not virtually abelian, then it contains a non abelian free group made of automorphisms
with positive entropy (see \cite{Cantat:Milnor}); each of these automorphisms has a singular measure of maximal entropy.
\end{rem}

\subsubsection{Kummer surfaces} 

Building on the strategy of \cite{Cantat:Compositio}, it seems reasonable to expect the existence of 
Kummer surfaces $X$ such that 
(i) $X$ is a K3 surface, (ii) there is an automorphism
$f$ of $X$ of positive entropy with $\mu_f=\Omega_X\wedge{\overline{\Omega_X}}$, and (iii) there is 
an automorphism $g$ of $X$ of positive entropy such that $\mu_g$ is singular.

\subsection{Rational surfaces and Galois conjugates}

\subsubsection{Proof of Corollary \ref{thm:dyna-deg}}\label{par:egbdk}

Let $A$ be a complex abelian surface. Its Picard number is bounded from above
by $h^{1,1}(A;\R)$ hence by $4$. The dynamical degree $\lambda_g$ of every $g \in Aut(A)$ is the largest eigenvalue of $g^*$ on 
$\NS(A)\otimes_\Z \R$. As such, $\lambda_g$ is a root of the characteristic
polynomial of $g^*\colon \NS(A)\to \NS(A)$, and it is an algebraic integer of degree at most $4$. Passing to a finite
$g$-equivariant quotient $\pi\colon A\to X_0$, one does not change the topological entropy. Thus, if $(X,f)$ is a Kummer
example, the dynamical degree $\lambda_f$ is also an algebraic integer of degree at most $4$.

\begin{eg}
In \cite{Bedford-Kim:2010,Bedford-Kim:2012} and \cite{McMullen:2007}, Bedford and Kim and McMullen construct examples of automorphisms $f_n\colon X_n\to X_n$
of rational surfaces with positive entropy. Corollary \ref{thm:dyna-deg} shows that most of them 
have a singular measure of maximal entropy. For instance, with the notation of \cite{McMullen:2007}, the dynamical
degrees $\lambda_n$ of the ``Coxeter automorphisms'' form a sequence of Salem numbers that converges towards
the smallest Pisot number (a root of $\theta^3=\theta + 1$). This implies that the degree of $\lambda_n$ goes to 
$\infty$ with $n$. Note that the existence of automorphisms with a singular measure of maximal entropy had
already been observed in  \cite{McMullen:2007}, chapter 9.
\end{eg}

\subsubsection{Blanc's automorphisms}\label{par:Blanc-auto}

\begin{lem}\label{lem:ellcurve}
Let $f\colon X \to X$ be an automorphism of a complex projective surface with positive entropy $\log \lambda_f$. 
Assume that $X$ contains a curve of genus~$1$ which is periodic under the action of $f$. Then $\mu_f$ is singular
with respect to Lebesgue measure. 
\end{lem}

\begin{proof}
An automorphism of an abelian surface with positive entropy does not preserve any curve of genus~$1$. 
The lemma follows from this remark and our main theorem. 
\end{proof}

Let us now describe a construction due to Blanc (see \cite{Blanc:2008}). 

Consider a smooth plane cubic curve $C$. Given a point $q$ on $C$, there
is a unique birational involution $\sigma_q\colon \P^2_\C\dasharrow \P^2_\C$ that fixes $C$ point-wise and  
preserves the pencil of lines through $q$. The indeterminacy points of $\sigma_q$ are $q$ itself and the four
points $p$ of $C$ such that the line $(pq)$ is tangent to $C$ at $p$. Blowing up these points, one can lift
$\sigma_q$ to an automorphism of a rational surface $X_q$. The strict transform of $C$ is fixed point-wise
by this automorphism. 

Now, do that for $\ell$ points $q_i$ on the cubic $C$. Since $C$ is fixed point-wise by the
involutions $\sigma_{q_i}$, one can lift them simultaneously as automorphisms on the same surface (blowing up
$5\ell$ points). This provides an example of a rational surface $X_\ell$ with a huge group of automorphisms: Blanc 
proved that there is no relation between the involutions, they generate a subgroup of $\Aut(X_\ell)$ isomorphic
to the free product of $\ell$ copies of $\Z/2\Z$ (see \cite{Blanc:2008}, Theorem 6). 

There is a meromorphic $2$-form $\Omega_{\ell}$ on $X_\ell$ that does not vanish and 
has poles of order $1$ along the strict transform $C'$ of $C$ (two such forms are proportional). 

Now, define $f_l\colon X_\ell\to X_\ell$ to be the composition $\sigma_{q_1}\circ \ldots \circ \sigma_{q_\ell}$ of
the $\ell$ involutions. If $\ell \geq 3$, one gets an automorphism of $X_\ell$ of positive entropy.
The meromorphic $2$-form $\Omega_\ell$ determines a volume form $\Omega_\ell\wedge \overline{\Omega_\ell}$
with poles along $C'$; this form has infinite volume, as for
\[
\frac{\ii}{2}\frac{dz\wedge d\bar z}{z\bar z}
\]
near the origin in $\C$.  Thus, the transformation $f$ preserves a ``meromorphic'' volume form of infinite volume. 
Since $f$ fixes the strict transform of a smooth cubic curve, Lemma~\ref{lem:ellcurve} implies the following proposition.
\begin{pro}\label{blanc}
Let $f_l$ be Blanc's example described above, with $l\geq 3$. Then $\mu_f$ is singular with respect to Lebesgue measure. 
\end{pro}

\section{Appendix} \label{par:appendixwhole}

\begin{center}
{{(after Dinh-Sibony and Moncet)}}
\end{center}

\vspace{0.1cm}

{\small{

In this appendix, $f$ is an automorphism of a complex projective (resp. compact K\"ahler) surface $X$ with positive entropy $\log\lambda_f$, and $X_0$ is the surface obtained by blowing down all periodic curves of $f$.
Most of the results described in this appendix are due to Dinh, Sibony, and Moncet (see \cite{Dinh-Sibony:preprint} and \cite{M}).

\subsection{Contraction of periodic curves} \label{appendix:contraction}

The notations are as in Section~\ref{par:prelim}. The tensor product $\NS(X)\otimes_\Z A$, for $A$ a ring,  is denoted $\NS(X;A)$.

\subsubsection{Periodic curves and the space $\NN_f$}
Consider the subspace $\PC_f$ of $\NS(X; \Q)$ generated by all classes of $f$-periodic curves. We consider it
as a vector subspace of $\NS(X;\Q)$, and we denote by $\PC_f(\Z)$ its integral points. By construction, 
$\PC_f$ is a subspace of $\Pi_f^\perp$, because $\langle \theta^\pm_f\vert [C]\rangle=0$ for every periodic curve.

Denote by $\Psi_f\in \Z[t]$ the minimal polynomial of the algebraic integer $\lambda_f$. The characteristic polynomial
of $f^*\in \GL(\NS(X))$ is a product of $\Psi_f$ and cyclotomic factors. The vector space $\NS(X;\Q)$ splits as a direct
sum $\NN_f\oplus \NN_f^\perp$ such that 
\begin{itemize}
\item[(1)] $\NN_f$ and $\NN_f^\perp$ are $f^*$-invariant;
\item[(2)] the characteristic polynomial of $f^*\colon \NN_f\to \NN_f$ is equal to $\Psi_f$;
\item[(3)] the characteristic polynomial of $f^*\colon \NN_f^\perp\to \NN_f^\perp$ is a product of cyclotomic
factors;
\item[(4)] $\NN_f$ contains $\theta^+_f$ and $\theta^-_f$, and the intersection form is of Minkowski type
on~$\NN_f$; 
\item[(5)] $\NN_f^\perp$ contains $\PC_f$ and the intersection form is negative definite on $\NN_f^\perp$.
\end{itemize}

\begin{rem}
If $\PC_f=\Pi_f^\perp$ then the plane $\Pi_f$ is defined over $\Q$ and  $\lambda_f$ is a quadratic unit, so that 
$\PC_f$ is, usually, much smaller than $\Pi_f^\perp$. (see \cite{Cantat-Lamy:Acta}). 
\end{rem}

\begin{lem} Let $E\subset X$ be a smooth elliptic  curve. If $E$ is $f$-periodic the image of the restriction morphism
\[
{\mathrm{res}}_E\colon \NN_f\to \Pic (E)
\]
is finite.
\end{lem}

\begin{proof}
Replacing $f$ by a positive iterate, one assumes that $E$ is $f$-invariant. The automorphism $f$ determines
a holomorphic automorphism $f^\sharp$ of the group $\Pic^0(E)$; since every holomorphic 
automorphism of an elliptic curve has finite order, we may assume that the automorphism 
$f^\sharp\colon \Pic^0(E)\to \Pic^0(E)$ is the identity. 

Every class $c$ in $\NN_f$ intersects $E$ trivially: $\langle [E]\vert c\rangle=0$, because $\NN_f$ is orthogonal to $\PC_f$. 
Thus the image of ${\mathrm{res}}_E$ is contained in $\Pic^0(E)$. Since ${\mathrm{res}}_E$ is equivariant under
the action of $f^*$ on $\NN_f$ and the action of $f^\sharp$ on $\Pic^0(E)$, we get 
\[
{\mathrm{res}}_E\circ f^*={\mathrm{res}}_E.
\]
Consequently, the kernel of ${\mathrm{res}}_E$ contains the image of $f^*-{\mathrm{Id}}_{\NN_f}$; but this image
has finite index in $\NN_f(\Z)$ because $1$ is not a root of $\Psi_f$.
\end{proof}

Let us now modify $X$ by a finite sequence of equivariant blow-ups to assume that every irreducible 
$f$-periodic curve is smooth. Thus, each of these curves is either a rational curve or a smooth elliptic curve
(see \S~\ref{contraper}).

\begin{lem}\label{lem:rest-trivial}
There is a finite index subgroup $\NN_f'$ of $\NN_f(\Z)$ such that every line bundle $L$ with 
first Chern class in $\NN_f'$ satisfies $L_{\vert E}={\mathcal{O}}_E$ for every $f$-periodic 
irreducible curve $E$. 
\end{lem}

\begin{proof} 
If $F$ is a rational periodic curve and if ${\mathrm{c}}_1(L)$ is an element of $\NN_f$, the  degree of $L_{\vert F}$ is equal to $0$; hence, $L_{\vert F}={\mathcal{O}}_F$.
Consequently, we can define $\NN'_f$ as the intersection of the kernels of ${\mathrm{res}}_E$,  where $E$ describe the finite set
of elliptic periodic curves.
 \end{proof}

\subsubsection{Big and nef classes (see \cite{Lazarsfeld})}\label{par:bignef}
The {\bf{pseudo-effective cone}} is the closure, in $\NS(X;\R)$, of the set of classes of effective $\R$-divisors. 
A class $u\in \NS(X;\R)$ is {\bf{big}} if it is in the interior of the pseudo-effective cone. 
A class $u\in \NS(X;\R)$ is big and nef if and only if it is nef and satisfies $u^2>0$. The set of big and 
nef classes forms a convex cone that contains the K\"ahler cone.

If the Chern class of a line bundle $L$ is big and nef, then the Kodaira dimension of $(X,L)$ is equal to $2$ (see \cite{Fujita, Lazarsfeld}). 

Let $u$ be a class with positive self-intersection. It is nef if and only if it intersects every irreducible curve $C\subset X$ non-negatively. In other words, the boundary of the big and nef cone is bounded by the quadratic cone $\{u^2=0\}$ and by hyperplanes 
\[
\langle u\vert [C]\rangle=0
\]
where $C$ is an irreducible curve (one may need infinitely many hyperplanes of this type to describe the boundary of the cone). 

\subsubsection{$X_0$ is projective}

\begin{thm}\label{thm:X0proj} Let $f$ be an automorphism of a projective surface $X$, with dynamical degree $\lambda_f>1$. Let 
$\pi\colon X\to X_0$ be the birational morphism that contracts all periodic curves of $f$, and only those curves.
The (singular) surface $X_0$ is projective.
\end{thm}

Recall that $\theta^+_f$ and $\theta^-_f$ denote the co-homology classes of $T^+_f$ and $T^-_f$.

\begin{lem}
There is an open neighborhood ${\mathcal{W}}$ of the class $\theta^+_f+\theta^-_f$ in $\NN_f$ 
which is contained in the big and nef cone.  
\end{lem}

\begin{proof}
Denote by $\Sigma$ the sum $\theta^+_f+\theta^-_f$. It is nef, because $\theta^+_f$ and $\theta^-_f$ are
nef. Since $\Sigma^2=\langle\theta^+_f\vert \theta^-_f\rangle>0$,  it is also big. If $C$ is a(n effective) curve then $\langle \Sigma\vert [C]\rangle \geq 0$, 
with equality if and only if $[C]$ is in $\Pi_f^\perp$,
if and only if $C$ is a periodic curve, if and only if $[C]\in \PC_f$ (see \S~\ref{contraper}). 

We now prove the lemma by contradiction. Since the condition $u^2>0$ is open, we may
assume that there is a sequence $(w_n)$ of classes $w_n\in \NN_f$ converging towards $\Sigma$ such that $w_n$ is not nef.  
Since $w_n$ is not nef, there exists an irreducible curve $C_n$ such that 
$\langle w_n\vert [C_n]\rangle<0$ (see \S~\ref{par:bignef}). In particular, the curve $C_n$ is not in $\PC_f$, because $w_n$ is an element of $\NN_f$; thus,  $\langle \Sigma\vert  [C_n]\rangle> 0$. Similarly, we may assume that the curves $C_n$ are pairwise distinct; otherwise, we could 
extract a constant subsequence $C_{n_j}=C$ and we would have $\langle \Sigma\vert [C]\rangle\leq 0$ because $\Sigma$
is the limit of $(w_n)$: this would contradict the fact that $C$ is effective but not in $\PC_f$.

Take a subsequence of $([C_n]/\parallel [C_n]\parallel)$ that converges to a pseudo-effective class $c_\infty$. We have 
$\langle \Sigma\vert c_\infty\rangle=0$, because $w_n$ converges towards $\Sigma$. Being pseudo-effective, $c_\infty$ is 
in $\Pi_f^\perp$ and consequently $c_\infty^2<0$. On the other hand, 
$c_\infty^2$ is the limit of $\langle C_n\vert C_{n+1}\rangle/(\parallel [C_n]\parallel\parallel [C_m]\parallel)$ and, as such, is non-negative.
This contradiction concludes the proof. 
\end{proof}

\begin{rem}
Since the cone of nef and big classes is $f^*$-invariant, this lemma implies that every element of the form 
$a\theta^+_f+b\theta^-_f$ with $a$ and $b$ positive is in this cone, and is in the relative interior of this cone
in $\NN_f(\R)$.
\end{rem}

To prove Theorem~\ref{thm:X0proj}, we may assume that all periodic curves of $f$ are smooth (otherwise, resolve
the singularities by a finite, equivariant, sequence of blow-ups).

Consider the set of all line bundles $L$ such that ${\mathrm{c}}_1(L)$ is in $\NN_f'$ and is big and nef. The Kodaira
dimension ${\mathrm{kod}}(X;L)$ of such a line bundle is equal to $2$, but the linear system $\vert L\vert$
may have fixed components. Write 
\[
L=M+F
\]
where $M$ is the mobile part and $F$ is the fixed part; ${\mathrm{kod}}(X;M)$ is equal to $2$, and the linear 
system $\vert M\vert$ has no fixed component. Taking sums (i.e. tensor products) $L_1+L_2$ the fixed part
decreases. Thus, there is an effective divisor $R$ such that $R$ is contained in the fixed part of $\vert L\vert$ 
for all such line bundles $L$ and $R$ is equal to the fixed part of $L$ if $L$ is ``sufficiently large''. By construction, 
$R$ is uniquely determined by $\NN_f$ and is therefore $f$-invariant. Thus, $R$ is a sum (with multiplicities)
of periodic curves. 

Take $L$ with fixed part equal to $R$ and consider its mobile part $M$. The linear system $\vert M\vert$ 
may have base points. Again, we may choose $L$ such that the set of base points is $f$-invariant. Blow-up these
base points (including infinitely near points), to get a new surface $Y$ with a birational morphism
$\epsilon \colon Y\to X$. By construction $f$ lifts to 
an automorphism $f_Y$ of $Y$ and  all irreducible periodic curves of $f_Y$ are smooth, because they are either strict transforms of
periodic curves of $f$ or exceptional divisors. The line bundle $L$ lifts to a big and nef bundle $L_Y$ such  
\[
L_Y=M_Y+R_Y
\]
where the fixed part $R_Y$ is made of $f_Y$-periodic curves and the mobile part $M_Y$ is base point free. 
Since the Kodaira dimension ${\mathrm{kod}}(Y;M_Y)$ is equal to $2$, the linear system $\vert M_Y\vert$
determines a morphism $\eta\colon Y\to Z$: the linear system $\vert M_Y\vert$ corresponds to the linear system $\vert M_Z\vert$  of 
hyperplane sections of $Z$. On the surface $Z$, one gets $L_Z=M_Z+R_Z$ with, now, an ample mobile part. 
In particular, $H^1(Z, mM_Z)$ vanishes if $m$ is large enough (see \cite{Lazarsfeld}, \S 1.4.D and 4.3). 
Thus, the morphism 
\[
H^0(Z; L_Z^{\otimes m})\to H^0(R_Z; L^{\otimes m}_{Z\vert R_Z})
\]
is onto if $m$ is large. Since, by Lemma~\ref{lem:rest-trivial},  the restriction of $L_Z$ to each irreducible component of $R_Z$ is trivial, 
there is a section of $L_Z^{\otimes m}$ which does not vanish on $R_Z$. Hence, $L_Z^{\otimes  m}$ has no
fixed component and is ample. On the other hand, $L$ (and thus $L_Z$) intersects trivially all periodic curves
of $f$ (resp. of $f_Z$). This implies that $\eta$ contracts all periodic curves of $f_Y$; in particular, 
$\eta$ contracts all exceptional divisors of $\epsilon$ and $\eta\circ\epsilon^{-1}\colon X\to Z$ is a morphism
that contracts all periodic curves of $f$. Since $Z$ is a projective surface,  $X_0$ is also
a projective surface. 

\begin{rem}\label{rem:konX0}
Once all periodic curves of $f$ have been contracted, the class $\Sigma$ becomes an ample class.
\end{rem}

\subsection{Continuous potentials}\label{par:subs-cont-pot}
Denote by   $\pi\colon X\to X_0$ the morphism that contracts all periodic 
curves of $f$ (and is an isomorphism in the complement of these curves). By the previous section, 
the surface $X_0$ is projective. Its N\'eron-Severi group $\NS(X_0)$ lifts to the orthogonal complement 
of $\PC_f$ in $\NS(X)$. Let $(u_i)$ be a basis of $\NS(X_0;\Q)$ such that every $u_i$ is ample. 
Each $u_i$ is the class of a K\"ahler form: up to a multiplicative factor, it is 
the restriction of the Fubini-Study form $\kappa_{FS}$ to $X_0$ for some embedding $\iota_i\colon X_0\to \P^{k_i}$ in a projective space; more precisely $u_i$ is the class of the form $\kappa_{i,0}=\iota_i^*\kappa_{FS}$. 
The pull-backs $\pi^*\kappa_{i,0}$ of these forms on $X$ give a finite set of smooth $(1,1)$-forms $\kappa_i$ which are
K\"ahler forms in the complement of the set of periodic curves, and vanish along them. 
Moreover, if $D$ is a connected component of the set of periodic curves, $\pi$ maps $D$ to 
a point $q$ of $X_0$; thus, each $\kappa_i$ is given by a local potential $w_i$ in some neighborhood
$\U$ of $D$. Moreover, adding a constant, $w_i$ vanishes identically on $D$. Now, use the fact that
the pre-image of  $\NS(X_0)$ in $\NS(X)$ is $f$-invariant, and apply $f^*$ to $\kappa_1$:
\[
\frac{1}{\lambda_f^n}(f^n)^*\kappa_1=\sum a_i(n)\kappa_i + dd^c(v_n)
\]
where the $a_i(n)$ are real numbers, $v_n\colon X\to \R$ is a smooth function, and $v_n$ is 
constant on every connected component of the set of periodic curves of $f$.
 
Then, apply the same strategy as in \cite{Dinh-Sibony:2005}:  the sequence ${\lambda_f^{-n}}(f^n)^*\kappa_1$ 
converges towards $T^+_f$;  each sequence of real numbers $a_i(n)$ converges to a real number $\alpha_i$; 
and $(v_n)$ converges uniformly to a H\"older continuous function $w_\infty$. Thus, 
\begin{eqnarray*}
T^+_f & = & dd^c\left(  \sum_i \alpha_i w_i + w_\infty \right) \\
& = & \sum_i \alpha_i \kappa_i + dd^c w_\infty,
\end{eqnarray*}
where the first expression is a local one and the second is global.
The function $w_\infty$ is constant on every connected component of the set of periodic curves of $f$; one can
therefore write $\pi_*(T^+_f)=\sum_i \alpha_i\kappa_{i,0} + dd^c(w_{\infty,0})$ for some continuous function $w_{\infty,0}$
on $X_0$.
There are analogous expressions for $T^-_f$. Similarly,  Remark~\ref{rem:konX0} shows that  
 \[
\pi_*(T^+_f+T^-_f)= \kappa + dd^c(w)
\]
where $\kappa$ is a k\"ahler form on $X_0$ and $w\colon X_0\to \R$ is a continuous function.

\begin{thm}\label{thm:continuous-potentials}
Let $f$ be an automorphism of a complex projective surface $X$, with positive entropy.
The current $T^+_f$ (resp. $T^-_f$) is a closed positive current with H\"older continuous potentials. If
$D$ is a connected component of the union of all $f$-periodic curves, there exists a neighborhood
${\mathcal{U}}$ of $D$ and a H\"older continuous  function $u\colon \U\to \R$ such that 
$T^+_f=dd^c u$ on $\U$ and $u=0$ on $D$. The projection $\pi_*T^+_f$ of $T^+_f$ on $X_0$ is a closed positive
current with continuous potentials, and $\pi_*(T^+_f+T^-_f)$ is co-homologous to a K\"ahler form on $X_0$. 
\end{thm}

\subsection{Dinh-Sibony Theorem}\label{par:ds-appendix}

\subsubsection{Brody curves and Dinh-Sibony theorem}
Let $M$ be a compact K\"ahler manifold (or compact, complex analytic space). 
Fix a hermitian metric $\kappa$ on $M$. 
A {\bf{Brody curve}} $\xi\colon \C\to M$ is a non-constant entire curve such that $\parallel \xi'(z)\parallel_\kappa$ is
uniformly bounded. 

Let $\xi\colon \C\to M$ be a  non-constant entire curve, and denote by $A(r;\xi)$  the area of $\xi(\disk_r)$ and by 
$L(r;\xi)$ the length of $\partial\xi(\disk_r)$; more precisely,
\begin{eqnarray*}
A(r;\xi) & = & \int_{\xi(\disk_r)}\kappa \, = \, \int_0^r\int_0^{2\pi} \parallel \xi'(te^{\ii\theta})\parallel^2 tdtd\theta\\
L(r;\xi) & = & \int_0^{2\pi} \parallel \xi'(re^{\ii \theta})\parallel rd\theta.
\end{eqnarray*}
The currents 
\[
S_r=\frac{1}{A(r;\xi)}\{\xi(\disk_r)\}
\]
form a family of positive currents of mass $1$; the mass of their boundary is equal to the ratio $L(r;\xi)/A(r;\xi)$.
An {\bf{Ahlfors current}} for the curve $\xi$ is a closed positive current obtained as a limit of the currents $S_{r}$ 
along a subsequence $(r_n)$, with $r_n\to_n \infty$; Ahlfors currents always exist (see \cite{Brunella:1999}).

\begin{thm}[Dinh-Sibony, see \cite{Dinh-Sibony:preprint}]
Let $M$ be a compact K\"ahler manifold (resp. a compact, K\"ahler analytic space).
 Let $T$ be a closed positive $(1,1)$-current with continuous
potentials. Let $\xi\colon \C\to X$ be a Brody curve, such that 
$
\xi^*T=0.
$
Then there is at least one  Ahlfors current $S$ associated to $\xi$ such that $T\wedge S =0$.
\end{thm}

In what follows, we shall apply this result in the singular surface $X_0$.

\subsection{Applications}\label{app-appendix}

\subsubsection{Proof of Corollary~\ref{cor:DS}} 
The first property follows from the fact that $\pi_*(T^+_f+T^-_f)$ is co-homologous to a K\"ahler form. More precisely, Section~\ref{par:subs-cont-pot} shows that $\pi_*(T^+_f + T^-_f)$ is
equal to the sum of a k\"ahler form $\kappa$ plus $dd^c(w)$ for some globally defined continuous 
function. Hence, the total mass of the product measure $\pi_*(T^+_f+T^-_f)\wedge A_\nu$ is equal to the mass of $\kappa \wedge A_\nu$ and, therefore, to the mass of the current $A_\nu$ with respect to $\kappa$;  as such, it is positive.

To prove the second property, assume that such a curve $\xi$ exists. If it is not a Brody curve, apply Zalcman's re-parametrization Lemma and Lemma~\ref{lem:pull-back-cv} to change it in a Brody curve. Then, apply Dinh-Sibony theorem to construct a Ahlfors current that contradicts the first property.
 
\subsubsection{Fatou components} 

Let $M$ be a compact, complex analytic space, with a fixed hermitian metric. Let $\U$ be a 
subset of $M$. One says that $\U$ is {\bf{hyperbolically embedded}} in $M$ if there is
a uniform bound for the derivative $\parallel \varphi'(0)\parallel$ of all holomorphic
mappings $\varphi\colon \disk\to \U\subset M$. Brody's Lemma implies that there exists
a Brody curve $\xi\colon \C\to {\overline{\U}}$ if $\U$ is not hyperbolically embedded. If
$\U$ is hyperbolically embedded, it is Kobayashi hyperbolic, meaning that the Kobayashi
pseudo-distance is a distance (see \cite{Lang}). 

Let $f$ be an automorphism of a complex projective surface $X$ with $\lambda_f>1$. 
The Fatou set $\Fat(f)$ is the largest open subset on which the sequence $(f^n)_{n\in \Z}$ is
locally equicontinuous. 

Let $K(f)$ be the support of $T^+_f+T^-_f$, and let $U(f)$ be its complement.  
It is easy to show that $\Fat(f)$ is contained in $U(f)$ (see \cite{Ueda:1994, Dinh-Sibony:preprint,M}).

Theorem~\ref{thm:continuous-potentials} shows that $\pi_*(T^+_f+T^-_f)$ has continuous potentials on $X_0$.
If $\pi(U(f))\subset X_0$ is not hyperbolically 
embedded in $X_0$, Brody re-parametrization lemma and 
Lemma~\ref{lem:pull-back-cv} provide  a Brody curve $\xi\colon \C\to   X_0$ such that $\xi^*(\pi_*(T^+_f+T^-_f))=0$. 
Dinh-Sibony Theorem asserts that there is an Ahlfors current $S$ for $\xi$ such that 
$\pi_*(T^+_f+T^-_f)\wedge S=0$. This contradicts Corollary~\ref{cor:DS}. Thus,  $\pi(U(f))$ is hyperbolically embedded in $X_0$. 

\begin{thm}[Dinh-Sibony, Moncet]
Let $f$ be an automorphism of a complex projective surface $X$ with positive entropy.
The Fatou set $\Fat(f)$ is Kobayashi hyperbolic modulo periodic curves: if $x$ and $y$ are distinct
points of $\Fat(f)$ with Kobayashi distance ${\mathrm{kobdist}}(x,y)=0$, then $x$ and $y$ are contained 
in a (reducible) connected periodic curve of $f$.
\end{thm}

\begin{proof}
The projection $\pi\colon X\to X_0$ is holomorphic and every holomorphic map is $1$-Lipschitz with respect
to the Kobayashi distance. Hence, $\pi(x)=\pi(y)$, because $\pi(\U)$ is Kobayashi hyperbolic. 
\end{proof}
}}

\bibliographystyle{plain}
\bibliography{referencessmooth}

\end{document}